\documentclass[a4paper, 11pt]{article}

\usepackage[arxiv]{optional}
\usepackage[english]{babel}
\usepackage[T1]{fontenc}
\usepackage{amsmath}
\opt{arxiv}{\usepackage{amsthm}
}\usepackage{amsfonts}
\usepackage{amssymb}
\usepackage[final]{graphicx}
\usepackage{paralist}
\usepackage{hyperref}
\usepackage{tikz}
\usepackage{booktabs}

\usepackage{cleveref}
\usepackage{subfigure}
\usepackage{pifont}
\usepackage{caption}
\usepackage{mathtools}

\opt{draft}{\usepackage{refcheck}
}

\opt{arxiv}{\theoremstyle{plain}
\newtheorem{theorem}{Theorem}[section]
\newtheorem{proposition}[theorem]{Proposition}
\newtheorem{lemma}[theorem]{Lemma}
\newtheorem{corollary}[theorem]{Corollary}
\newtheorem{conject}[theorem]{Conjecture}

\theoremstyle{definition}
\newtheorem{definition}[theorem]{Definition}
\newtheorem{example}[theorem]{Example}
\newtheorem{remark}[theorem]{Remark}
}\opt{springer}{

}

\renewcommand{\emptyset}{\varnothing}

\newcommand{\ee}{{\operatorname{e}}}
\newcommand{\ii}{{\operatorname{i}}}
\renewcommand{\i}{\textrm{i}}

\DeclareMathOperator{\im}{Im}
\DeclareMathOperator{\ind}{ind}

\DeclareMathOperator{\re}{Re}

\DeclareMathOperator{\wind}{wind}

\opt{springer}{
}
\opt{arxiv}{
}

\DeclarePairedDelimiter{\abs}{\lvert}{\rvert}

\DeclarePairedDelimiter\ceil{\lceil}{\rceil}
\DeclarePairedDelimiter\floor{\lfloor}{\rfloor}

\newcommand{\R}{\ensuremath{\mathbb{R}}}
\newcommand{\C}{\ensuremath{\mathbb{C}}}

\newcommand{\N}{\ensuremath{\mathbb{N}}}

\newcommand{\cH}{\mathcal{H}}
\newcommand{\cO}{\mathcal{O}}
\newcommand{\cP}{\mathcal{P}}

\newcommand{\conj}[1]{\overline{#1}}

 \numberwithin{equation}{section}

\title{\textbf{On the zeros of polyanalytic polynomials}}

\author{Olivier S\`ete\footnotemark[1] \and Jan Zur\footnotemark[3]}

\date{June 13, 2024}

\begin{document}
\maketitle

\renewcommand{\thefootnote}{\fnsymbol{footnote}}

\footnotetext[1]{TU Berlin, Institute of Mathematics, MA 3-3, Stra{\ss}e des 
17. Juni 136, 10623 Berlin, Germany,
\texttt{sete@math.tu-berlin.de}, ORCID: 0000-0003-3107-3053}

\footnotetext[3]{TU Berlin, Institute of Mathematics, MA 3-3, Stra{\ss}e des 
17. Juni 136, 10623 Berlin, Germany,
\texttt{zur@math.tu-berlin.de}, ORCID: 0000-0002-5611-8900}

\begin{abstract}
We give sufficient conditions under which a polyanalytic polynomial of degree 
$n$ has (i) at least one zero, and (ii) finitely many zeros.  In the latter 
case, we prove that the number of zeros is bounded by $n^2$.
We then show that for all $k \in \{0, 1, 2, \dots, n^2, \infty\}$ there exists 
a 
polyanalytic polynomial of degree $n$ with exactly $k$ distinct zeros.
Moreover, we generalize the Lagrange and Cauchy bounds from analytic to 
polyanalytic polynomials and obtain inclusion disks for the zeros.
Finally, we construct a harmonic and thus polyanalytic polynomial of degree $n$ 
with $n$ nonzero coefficients and the maximum number of $n^2$ zeros.
 \end{abstract}


\paragraph*{Mathematics Subject Classification\,(2020):} 
30C55; 31A05; 30C10; 30C15. 

\section{Introduction}\label{sect:intro}

A \emph{polyanalytic function} is a complex-valued function of the form
\begin{equation}\label{eqn:pa_function}
F(z,\conj{z}) = \sum_{k=0}^m f_k(z)\conj{z}^k,
\end{equation}
where $f_0, \ldots, f_m$ are analytic functions in an open set $\Omega 
\subseteq \C$.
Hence, $F$ can be interpreted as a polynomial in $\conj{z}$ 
with analytic functions as coefficients.  
The term polyanalytic is motivated by the property
$\partial_{\conj z}^{m+1} F = 0$, where $\partial_{z}$ and $\partial_{\conj z}$ 
are the Wirtinger derivatives.
In the sequel, we denote analytic functions by lower case letters, and 
polyanalytic or harmonic functions by capital 
letters.  Moreover, we use a notation common in the literature and write 
$F(z,\conj z)$ instead of $F(z)$ for polyanalytic and harmonic functions
although $z$ and $\conj{z}$ are not independent variables of $F$.

The polyanalytic function theory was strongly influenced by Balk and his 
students.  They developed analogues of several classical results for analytic 
functions, e.g., Liouville's theorem~\cite[Thm.~4.2]{Balk1991}, Picard's 
theorems~\cite[Sect.~5.2]{Balk1991}, Cauchy's integral formula and other 
integral representations~\cite[Thm.~1.4]{Balk1997}.
See~\cite{Balk1991,Balk1997} for an overview on their results.
Other researchers also contributed to the classical theory of polyanalytic 
functions, 
e.g., Bosch and Krajkiewicz~\cite{KrajkiewiczBosch1969,BoschKrajkiewicz1970}.
Besides the function theoretic approach, polyanalytic functions have also been 
studied in a functional analytic context; see, e.g., the survey of
Abreu and Feichtinger~\cite{AbreuFeichtinger2014} and references therein.

If $f_0, \ldots, f_m$ in~\eqref{eqn:pa_function} are analytic 
\emph{polynomials}, 
$F$ is called a \emph{polyanalytic polynomial}, which is a polynomial in $z$ 
and $\conj{z}$.  In the sequel, we will use its \emph{bivariate form}
\begin{equation}\label{eqn:papolynomial}
P(z,\conj{z}) = \sum_{k=0}^n \sum_{j = 0}^k \alpha_{j,k-j}z^j \conj{z}^{k-j}.
\end{equation}
Note that two polyanalytic polynomials are equal if and only if they 
have the same coefficients $\alpha_{j,k-j} \in \C$ in~\eqref{eqn:papolynomial}.
Unlike analytic polynomials, polyanalytic polynomials can have 
an empty, finite, or infinite set of zeros.
For example, $P_1(z,\conj{z}) = z - \conj{z}$ vanishes on 
the real axis, i.e., $P_1$ has an unbounded set of nonisolated zeros, 
$P_2(z,\conj{z}) = z\conj{z}-1$ vanishes on the unit circle, i.e., it has a 
bounded set of nonisolated zeros, 
$P_3(z, \conj{z}) = z \conj{z}$ has an isolated zero,
while $P_4(z,\conj{z}) = z\conj{z}+1$ has no zeros at all.

Obviously, any polyanalytic polynomial $P(z,\conj z)$~\eqref{eqn:papolynomial} 
can be identified with a 
system of two real polynomials $p$ and $q$ in the real variables $x$ and $y$,
\begin{equation} \label{eqn:real_polynomial_system}
\begin{cases}
p(x,y) = \re(P(x + \ii y,x - \ii y)), \\
q(x,y) = \im(P(x + \ii y,x - \ii y)),
\end{cases}
\end{equation}
and any system of two real polynomials $p$ and $q$ in the real variables $x$ 
and $y$ can be identified with the polyanalytic polynomial
\begin{equation}
P(z,\conj z) = p\left(\frac{z+\conj z}{2},\frac{z-\conj z}{2\i}\right) 
+ \i q\left(\frac{z+\conj z}{2},\frac{z-\conj z}{2\i}\right).
\end{equation}
So what is the advantage of the polyanalytic notation?  Although polyanalytic 
polynomials are not analytic, we can use numerous tools of
complex analysis, e.g., the argument principle and the visualization of 
complex-valued functions, as it was done before in several numerical 
applications, where the polyanalytic perspective has turned out to be 
favorable.  In particular, for determining exclusion and inclusion regions for 
the eigenvalues of a normal matrix~\cite{HuhtanenLarsen2002}, computing an 
orthonormal basis of generalized Krylov subspaces~\cite{Huhtanen2003},  
and bivariate root-finding~\cite{SorberVanBarelDeLathauwer2014}.

This paper is devoted to discuss the following questions concerning the zeros 
of polyanalytic polynomials in Sections~\ref{sect:bounds}--\ref{sect:extremal}:

\smallskip

\begin{compactenum}
\item When does a polyanalytic polynomial have at least one zero?
\item When does a polyanalytic polynomial have finitely many zeros?
\item What is the maximum number of zeros of a polyanalytic polynomial, 
provided that this number is finite?
\item Which numbers of zeros can occur for a polyanalytic polynomial of degree 
$n$?
\item Where are the zeros of a polyanalytic polynomial located?
\item How can a polyanalytic polynomial with the maximum number of zeros be 
constructed?
\end{compactenum}

\smallskip

\noindent Moreover, we consider \emph{harmonic polynomials} 
\begin{equation}\label{eqn:harmonic_polynomial}
P(z,\conj z) = p(z) + \conj{q(z)},
\end{equation}
where $p$ and $q$ are analytic polynomials with $\deg(p) = n$ and $\deg(q) = m$.
Every harmonic polynomial is also a polyanalytic polynomial.
The absence of mixed terms $\alpha_{j,k-j} z^j \conj{z}^{k-j}$ simplifies the 
investigation.
Wilmshurst~\cite{Wilmshurst1998}
proved that $P$ has at most $n^2$ zeros if $\lim_{z \to 
\infty}P(z,\conj z) = \infty$ holds.  Independently, Peretz and 
Schmid~\cite{PeretzSchmid1997} proved the same result using algebraic geometry.
Moreover, Wilmshurst conjectured for $n > m$ that an upper bound on the number 
of zeros of $P$ as in~\eqref{eqn:harmonic_polynomial} is
\begin{equation}\label{eq:wilmshurst_conjecture}
N(P) \le 3n-2 + m(m-1).
\end{equation}
For $m = n-1$ this gives $n^2$.  He also provided an example to show 
that this bound is sharp for $m = n-1$.  Khavinson and 
{\'S}wi{\c{a}}tek~\cite{KhavinsonSwiatek2003}, and Geyer~\cite{Geyer2008} 
settled the conjecture for $m = 1$, while Lee, Lerario, and 
Lundberg~\cite{LeeLerarioLundberg2015,Lundberg2023} proved that the conjecture 
is not true in general.  Determining the maximum number of zeros of 
harmonic polynomials depending on the degrees $n$ and $m$ is still an open 
problem; see~\cite[Sect.~2.2]{BeneteauHudson2018} 
and~\cite[Sect.~4]{Lundberg2023}.  
We will discuss an extension from the harmonic to the polyanalytic case in 
Section~\ref{sect:wilmshurst}.
 \section{Number of zeros}\label{sect:bounds}

In this section, we answer questions (i)--(iv) from the introduction.
First, we introduce the degrees of a polyanalytic polynomial.

\begin{definition}
For a polyanalytic polynomial $P(z,\conj{z}) = \sum_{k=0}^m p_k(z)\conj{z}^k$ 
with nonzero $p_m$, 
we define the \emph{degree} $\deg(P) = \max_{k}(\deg(p_k) + k)$, the 
\emph{analytic degree} $\deg_{z}(P) = \max_{k}\deg(p_k)$, and the 
\emph{anti-analytic degree} $\deg_{\conj{z}}(P) = m$.
\end{definition}

\begin{example}
The polyanalytic polynomial $P(z,\conj{z}) = z^3 + z^2\conj{z}^2$ 
has the degrees $\deg(P) = 4$, $\deg_z(P) = 3$, and $\deg_{\conj{z}}(P) = 2$.
\end{example}

The first question, ``When does a polyanalytic polynomial have at least one 
zero?'', was answered by Balk in~\cite{Balk1968} 
and~\cite[Thm.~4.1]{Balk1991} by his fundamental theorem of algebra for 
polyanalytic polynomials.

\begin{theorem}\label{thm:fta}
A polyanalytic polynomial $P$ with $\deg(P) > 2\deg_{z}(P)$ or $\deg(P) > 
2\deg_{\conj{z}}(P)$ has at least one zero.
\end{theorem}

Theorem~\ref{thm:fta} is sharp in the following sense.  There exists a 
polyanalytic polynomial $P$ with $\deg(P) \leq 2 \deg_z(P)$ and $\deg(P) \leq 2 
\deg_{\conj{z}}(P)$ that has no zeros, e.g., given $k, n \in \N$ with $k \leq n 
\leq 2k$,
\begin{equation}
P(z,\conj{z}) = (z + \conj{z} + \i)^{2k-n}(z\conj{z} + 1)^{n-k}
\end{equation}
has no zeros but $\deg(P) = n$ and $\deg_{z}(P)= k = \deg_{\conj{z}}(P)$; 
see~\cite[Rem.~4.2]{Balk1991}.

We give another sufficient condition for the existence of at least one zero in 
Theorem~\ref{thm:sufficient_condition}, and hence an alternative fundamental 
theorem of algebra for polyanalytic polynomials.
Recall that the winding of a continuous function~$f$ along a curve $\Gamma$ on 
which $f$ has no zeros is the change of the argument of $f$ along $\Gamma$ 
divided by $2 \pi$,
\begin{equation}
\wind(f; \Gamma) = \frac{1}{2 \pi} \Delta_\Gamma \arg(f(z));
\end{equation}
see, e.g.,~\cite[Sect.~2.3]{Balk1991}, \cite[Sect.~2.3]{Sheil-Small2002}, 
or~\cite[Sect.~3.4]{Wegert2012}.

\begin{theorem}[{\cite[Thm.~2.3]{Balk1991}, 
\cite[Sect.~2.3.12]{Sheil-Small2002}}] \label{thm:zero_continuous_function}
Let the function $f$ be continuous in the closed domain $D \subseteq \C$ 
bounded by a closed curve $\Gamma$.  If $f$ has no zeros in $D$, then $\wind(f; 
\Gamma) = 0$.
\end{theorem}

The converse of Theorem~\ref{thm:zero_continuous_function} does not hold, since 
there exists functions with zeros but winding $0$, e.g., $P(z, \conj{z}) = z 
\conj{z}$ on $D = \{ z \in \C : \abs{z} \leq 1 \}$.

\begin{theorem} \label{thm:sufficient_condition}
A polyanalytic polynomial
\begin{equation}
P(z, \conj{z}) = \sum_{j=0}^n \alpha_{j, n-j} z^j \conj{z}^{n-j} + P_{n-1}(z, 
\conj{z})
\end{equation}
with $n \geq 1$, $\deg(P_{n-1}) \leq n-1$, and such that there exists an
$\ell \in \{ 0, 1, \ldots, n \}$, $2 \ell \neq n$, with
\begin{equation} \label{eqn:dominant_coefficient}
\abs{\alpha_{\ell, n-\ell}} > \sum_{j=0, j \neq \ell}^n \abs{\alpha_{j, n-j}},
\end{equation}
has at least one zero.  In particular, a polyanalytic polynomial of the form
\begin{equation}
P(z, \conj{z}) = \alpha_{\ell, n-\ell} z^\ell \conj{z}^{n-\ell} + P_{n-1}(z, 
\conj{z})
\end{equation}
with 
$n \geq 1$, $\alpha_{\ell, n-\ell} \neq 0$, $2 \ell \neq n$, and $\deg(P_{n-1}) 
\leq n-1$ has at least one zero.
\end{theorem}

\begin{proof}
Let us write $P(z, \conj{z}) = z^\ell \conj{z}^{n-\ell} (\alpha_{\ell, n-\ell} 
+ g(z))$ with
\begin{equation*}
g(z) = 
\sum_{j = 0, j \neq \ell}^n \alpha_{j,n-j} \frac{z^j \conj{z}^{n-j}}{z^\ell 
\conj{z}^{n-\ell}} + \frac{P_{n-1}(z, \conj{z})}{z^\ell \conj{z}^{n-\ell}}.
\end{equation*}
Note that $g$ is not polyanalytic in general.  Since
\begin{equation*}
\abs{g(z)} \leq \sum_{j = 0, j \neq \ell}^n \abs{\alpha_{j, n-j}} + 
\cO(\abs{z}^{-1}) < \abs{\alpha_{\ell, n-\ell}}
\end{equation*}
holds for all $z$ on a sufficiently large circle $\Gamma$ about $0$ (oriented 
in the positive sense), we have
$\wind(\alpha_{\ell, n-\ell} + g(z); \Gamma) = 0$ and therefore
\begin{align*}
\wind(P; \Gamma) &= \wind(z^\ell; \Gamma) + \wind(\conj{z}^{n-\ell}; \Gamma)
+ \wind(\alpha_{\ell, n-\ell} + g(z); \Gamma) \\
&= \ell - (n-\ell) = 2 \ell - n \neq 0.
\end{align*}
By Theorem~\ref{thm:zero_continuous_function}, $P$ has a zero in the interior 
of $\Gamma$.
\end{proof}

\begin{remark}
\begin{enumerate}
\item Theorem~\ref{thm:sufficient_condition} is sharp in the following sense.
There exists a polyanalytic polynomial $P$ which does not 
fulfill~\eqref{eqn:dominant_coefficient} and has no zeros, e.g., $P(z, 
\conj{z}) = z + \conj{z} + \ii$.
If $2 \ell = n$ in Theorem~\ref{thm:sufficient_condition}, then $P$ does not 
necessarily have a zero, e.g., for $P(z, \conj{z}) = z^\ell \conj{z}^\ell + 1$.

\item Neither Theorem~\ref{thm:fta} nor Theorem~\ref{thm:sufficient_condition} 
is covered by the other.
While Theorem~\ref{thm:fta} does not apply to $P(z, \conj{z}) = z^n + z 
\conj{z}^{n-2} + 1$, $n \geq 4$, Theorem~\ref{thm:sufficient_condition} shows 
that $P$ has at least one zero.
In contrast, for $P(z, \conj{z}) = z^n + z^\ell \conj{z}^{n-\ell} + 1$ with $n 
\geq 1$ and $n/2 < \ell \leq n$, Theorem~\ref{thm:fta} guarantees the existence 
of a zero, while Theorem~\ref{thm:sufficient_condition} does not apply.

\item Both Theorems~\ref{thm:fta} and~\ref{thm:sufficient_condition} imply the 
fundamental theorem of algebra for analytic polynomials.
\end{enumerate}
\end{remark}

Next, we consider the second question: ``When does a polyanalytic polynomial 
have finitely many zeros?''.  We start with the following observation; 
see~\cite[Rem.~4.3]{Balk1991} and~\cite[Sect.~2.2]{Balk1991}.

\begin{theorem} \label{thm:isolated_finite}
Let $P$ be a polyanalytic polynomial.
\begin{enumerate}
\item If the zeros of $P$ are isolated, then their number is finite.
\item If $P$ has nonisolated zeros, then $P$ vanishes on an arc.
\end{enumerate}
\end{theorem}

A polyanalytic polynomial with infinitely many zeros vanishes 
on some arc in $\C$.  If the polynomial is irreducible we have a necessary 
condition for this; see~\cite[p.~20]{Davis1974}.  Recall that a non-constant
polyanalytic polynomial is \emph{irreducible}, if it cannot be factored into 
the product of two non-constant polyanalytic polynomials, and \emph{reducible} 
otherwise.

\begin{theorem}
Let $P$ be an irreducible polyanalytic polynomial.  If $P$ has infinitely 
many zeros, then
\begin{equation}\label{eqn:self_conjugate}
P(z,\conj{z}) = \lambda \conj{P(z,\conj{z})}
\end{equation}
holds for some $\lambda \in \C$ with $\abs{\lambda} = 1$.
\end{theorem}

A polynomial $P$ satisfying~\eqref{eqn:self_conjugate} is called 
\emph{self-conjugate}~\cite[p.~21]{Davis1974}.  
We obtain the following sufficient condition for $P$ having finitely many zeros.

\begin{theorem} \label{thm:suff_finite_zeros}
Let $P(z, \conj{z}) = \sum_{j+k \leq n} \alpha_{j,k} z^j \conj{z}^k$ be an 
irreducible polyanalytic polynomial with $\abs{\alpha_{j,k}} \neq 
\abs{\alpha_{k,j}}$ for some $j,k$.
Then $P$ has finitely many zeros.
\end{theorem}

Next, we show that polyanalytic polynomials of the form $z^n + P_{n-1}(z, 
\conj{z})$ have a factorization into factors of the same form, and thus only 
finitely many zeros.
Related results for bivariate polynomials $y^n + \sum_{k=0}^{n-1} p_k(x) y^k$ 
over an algebraically closed field are 
obtained in~\cite{PanaitopolStefanescu1990}.  These polynomials are called 
\emph{generalized difference polynomials}, while $p(x)-q(y)$ is called a
\emph{difference polynomial}.  The latter is the bivariate 
pendant of a harmonic polynomial.
The next result can be found in~\cite[Sect.~1.2.5]{Sheil-Small2002}.
For completeness, we give an alternative proof.

\begin{theorem} \label{thm:finite_zeros}
Let $P(z,\conj{z}) = z^n + P_{n-1}(z,\conj{z})$ be a polyanalytic polynomial 
with $\deg(P_{n-1}) \le n-1$.  Then the following holds.
\begin{enumerate}
\item \label{it:factorization}
$P$ has a factorization into irreducible polyanalytic polynomials of the 
form $z^k + P_{k-1}(z, \conj{z})$ with $\deg(P_{k-1}) \leq k-1$.
\item \label{it:finitely_many_zeros}
$P$ has finitely many zeros.
\end{enumerate}
\end{theorem}

\begin{proof}
\ref{it:factorization}
If $P$ is reducible, then $P = P_\ell P_m$, where $\ell = \deg(P_\ell) 
\ge 1$ and $m = \deg(P_m) \ge  1$ with $\ell + m = n$.  We write
\begin{equation*}
P_\ell(z,\conj{z})
= \sum_{j=0}^\ell a_j z^{\ell - j} \conj{z}^j + P_{\ell-1}(z,\conj{z}), \quad
P_m(z,\conj{z}) = \sum_{k=0}^m b_k z^{m-k} \conj{z}^k + P_{m-1}(z,\conj{z}),
\end{equation*}
where $\deg(P_{\ell-1}) \leq \ell-1$ and $\deg(P_{m-1}) \leq m-1$.  The terms 
of degree $n$ of $P$ can only be obtained by multiplication of terms of degree 
$\ell$ of $P_\ell$ and terms of degree $m$ of $P_m$.  Comparing the terms of 
degree $n$ of $P$ and $P_\ell P_m$ gives
\begin{equation*}
z^n = \Bigg( \sum_{j=0}^\ell a_j z^{\ell - j} \conj{z}^j \Bigg)
\Bigg( \sum_{k=0}^m b_k z^{m - k} \conj{z}^k \Bigg)
= \sum_{k=0}^n \Bigg( \sum_{j=0}^k a_j b_{k-j} \Bigg) z^{n-k} \conj{z}^k.
\end{equation*}
Comparing coefficients yields
\begin{equation*}
a_0 b_0 = 1 \quad \text{and} \quad \sum_{j=0}^k a_j b_{k-j} = 0 \quad
\text{for} \quad k = 1, \ldots, n.
\end{equation*}
The auxiliary polynomials $a(t) = \sum_{j = 0}^\ell a_j t^j$ and $b(t) = 
\sum_{k=0}^m b_k t^k$ satisfy $a(t) b(t) = 1$, hence $a$ and $b$ are constant 
nonzero polynomials, and
\begin{equation*}
P_\ell(z,\conj{z}) = a_0 z^\ell + P_{\ell-1}(z,\conj{z}), \quad
P_m(z,\conj{z}) = b_0 z^m + P_{m-1}(z,\conj{z}).
\end{equation*}
Using $a_0 b_0 = 1$, we obtain the factorization
\begin{equation*}
P(z, \conj{z}) = ( z^\ell + a_0^{-1} P_{\ell-1}(z, \conj{z}) ) ( z^m + 
b_0^{-1}P_{m-1}(z, \conj{z}) ).
\end{equation*}
Inductively, $P$ is a product of irreducible polyanalytic polynomials of the 
form $z^k + P_{k-1}(z, \conj{z})$ with $\deg(P_{k-1}) \leq k-1$.

\ref{it:finitely_many_zeros} 
By (i), $P$ is the product of irreducible polynomials of the form $z^k + 
P_{k-1}(z, \conj{z})$ with $\deg(P_{k-1}) \leq k-1$, and by 
Theorem~\ref{thm:suff_finite_zeros}, each of these factor has only finitely 
many zeros, since $\alpha_{k,0} = 1 \neq 0 = \alpha_{0,k}$.
\end{proof}

The above theorem is stated for $P(z,\conj{z}) = z^n + P_{n-1}(z,\conj{z})$.  
For $Q(z,\conj{z}) = \conj{z}^n + P_{n-1}(z,\conj{z})$ we get the 
same result by considering $\conj{Q(z,\conj{z})}$.  The condition in 
Theorem~\ref{thm:finite_zeros} is sharp in the sense that we can not weaken the 
assumptions in terms of the monomials $z^k\conj{z}^{n-k}$ of degree $n$.
Two monomials of degree $n$ with nonzero coefficients allow infinitely many 
zeros, e.g., for $0 \leq k < \ell \leq n$, the polyanalytic polynomial 
\begin{equation}
P(z,\conj{z}) = z^k\conj{z}^{n-k} + z^\ell\conj{z}^{n-\ell} = z^k (z^{\ell-k}
+ \conj{z}^{\ell-k}) \conj{z}^{n-\ell}
\end{equation}
has infinitely many zeros.  
Moreover, a mixed monomial of degree $n$ may also lead to infinitely many 
zeros, e.g., for $0<k<n$, the polyanalytic polynomial
\begin{equation}
P(z,\conj{z}) = z^k\conj{z}^{n-k} - z^{k-1}\conj{z}^{n-k-1}
= z^{k-1}\conj{z}^{n - k-1}(z\conj{z}-1)
\end{equation}
vanishes on the unit circle.

To answer the third question, we follow the lines of Wilmshurst and use a 
special version of B\'ezout's theorem; 
see, e.g.,~\cite[Thm.~2]{Wilmshurst1998}, \cite[p.~204]{PeretzSchmid1997}, 
or~\cite[1.2.2]{Sheil-Small2002}.

\begin{theorem}\label{thm:bezout}
Let $p$ and $q$ be real polynomials in the real variables $x$ and $y$ with 
$\deg(p) = n$ and $\deg(q) = m$.  Then $p$ and $q$ have either at most $m n$ 
common zeros or have infinitely many common zeros.
\end{theorem}

This gives immediately the maximum number of zeros for polyanalytic polynomials 
with finitely many zeros.

\begin{theorem} \label{thm:upper_bound}
A polyanalytic polynomial $P$ of degree $n$ with finitely many zeros has at 
most $n^2$ zeros.
\end{theorem}

\begin{proof}
Applying Theorem~\ref{thm:bezout} to the equivalent 
system~\eqref{eqn:real_polynomial_system} gives the upper bound.
\end{proof}

If a polyanalytic polynomial is reducible, we give a tighter bound for the
number of zeros.  This yields also a sufficient condition for 
irreducibility of polyanalytic polynomials.

\begin{theorem} \label{thm:irreducible}
Let $P$ be a polyanalytic polynomial of degree $n$ with finitely many zeros.
\begin{enumerate}
\item \label{it:factors}
If $P$ is reducible, i.e., $P = P_{1}\cdot \ldots \cdot P_{k}$ with $\deg(P_j) 
= n_j \geq 1$ and $k \geq 2$, then $P$ has at most $\sum_{j = 1}^k n_j^2 < n^2$
zeros.

\item \label{it:criterion_irreducible}
Let $n \ge 2$. If $P$ has at least $n^2 - 2n + 3$ zeros, then $P$ is 
irreducible.
\end{enumerate}
\end{theorem}

\begin{proof}
\ref{it:factors} By Theorem~\ref{thm:upper_bound}, the factor $P_{j}$ has at 
most $n_j^2$ zeros.  Summing over all factors gives the assertion.

\ref{it:criterion_irreducible}
Assume that $P$ is reducible, i.e., $P = P_1 P_2$ with 
$\deg(P_1) = k > 0$ and $\deg(P_2) = n-k > 0$.
By~\ref{it:factors}, the number of zeros of $P$ is bounded by
$(n - k)^2 + k^2 < n^2 - 2n + 3$.
\end{proof}

\begin{corollary} \label{cor:wilmshurst}
\begin{enumerate}
\item Let $P(z, \conj{z}) = \alpha_{n,0} z^n + P_{n-1}(z, \conj{z})$ be a 
polyanalytic polynomial with $\alpha_{n,0} \neq 0$ 
and $\deg(P_{n-1}) \leq n-1$.
Then $P$ has at most $n^2$ zeros.

\item Let $P(z, \conj{z}) = p(z) + \conj{q(z)}$ be a harmonic polynomial with 
$\deg(p) = n > \deg(q)$.  Then $P$ has at most $n^2$ zeros.
\end{enumerate}
\smallskip
\noindent
Moreover, in both \textup{(i)} and \textup{(ii)}, if $P$ has $n^2$ zeros, then 
$P$ is irreducible.
\end{corollary}

\begin{proof}
By Theorem~\ref{thm:finite_zeros}, $P$ has finitely many zeros and, by 
Theorem~\ref{thm:upper_bound}, the number of zeros is at most $n^2$.
If $P$ has $n^2$ zeros, then $P$ is irreducible by 
Theorem~\ref{thm:irreducible}.
\end{proof}

\begin{remark}
\begin{enumerate}
\item The bounds in Theorem~\ref{thm:upper_bound} and 
Corollary~\ref{cor:wilmshurst} 
are sharp in the sense that for every $n$ there exists a harmonic, and in 
particular polyanalytic, polynomial of degree $n$ with $n^2$  zeros, e.g., the 
example of Wilmshurst 
\begin{equation}\label{eqn:wilmshurst_polynomial}
P(z,\conj z) = (z-1)^n + z^n + \conj{\ii (z-1)^n -\ii z^n};
\end{equation}
see~\cite[p.~2080]{Wilmshurst1998}.  Hence, mixed terms in $z$ and $\conj{z}$ 
are not necessary to obtain the maximum number of zeros.
While a polyanalytic polynomial of degree $n$ has $(n+1)(n+2)/2$ 
coefficients, only $2n$ coefficients in~\eqref{eqn:wilmshurst_polynomial} are 
nonzero.
In Section~\ref{sect:extremal}, we present another construction of a harmonic 
polynomial of degree $n$ with $n^2$ zeros, which, moreover, has only $n$ 
nonzero coefficients.

\item Another special class of polyanalytic polynomials of interest 
are \emph{logharmonic polynomials}
\begin{equation}
P(z,\conj z) = p(z)\conj{q(z)} - c,
\end{equation}
where $p$ and $q$ are analytic polynomials and $c \in \C$. 
Such a polynomial has at most $\deg(p)^2 + \deg(q)^2$ zeros provided that it is 
not self-conjugate, which was recently proven in 
\cite[Thm.~4.1]{KhavinsonLundbergPerry2024}.
\end{enumerate}
\end{remark}

Finally, we answer question (iv) from Section~\ref{sect:intro} on the 
possible numbers of zeros of polyanalytic polynomials.
The proof of the next result is based on previous work of the authors 
on harmonic functions~\cite{SeteZur2021}.

\begin{theorem}
Let $n \in \N$, $n \geq 1$.  For each $k \in \{0,1,2,\dots,n^2,\infty\}$ there 
exists a polyanalytic polynomial of degree $n$ with exactly $k$ zeros.
\end{theorem}

\begin{proof}
For $k = 0$, the polynomial $P(z, \conj{z}) = (z - \conj{z} - 1)^n = (2 \ii 
\im(z) - 1)^n$ has degree $n$ and no zeros.
For $k \in \{ 1, 2, \dots, n\}$, the polynomial $P(z,\conj{z}) = (z^k - 1)(z - 
\conj{z} - 1)^{n-k}$ has degree $n$ and exactly $k$ zeros.
Next, let $n \le k \le n^2$. By~\cite[Cor.~5.6]{SeteZur2021}, there exists a 
harmonic polynomial $P(z,\conj z) = p(z) + \conj{q(z)}$ with $\deg(p) = n >
\deg(q)$ and exactly $k$ zeros.  In particular, this is a polyanalytic 
polynomial of degree $n$.
For $k = \infty$, the polynomial $P(z,\conj{z}) = z^n - \conj{z}^n$ vanishes on 
the real axis.
\end{proof}
 \section{Location of zeros}
\label{sect:location}

In this section, we address the question where the zeros of a polyanalytic 
polynomial $P$ are located.  More precisely, we determine bounds for the 
maximum modulus of the zeros of $P$.  If $P$ has at least one zero, we denote
\begin{equation} \label{eqn:opt_radius}
r_* = \sup \{ \abs{z} : z \in \C, P(z, \conj{z}) = 0 \},
\end{equation}
i.e., all zeros of $P$ lie in the disk $\{ z \in \C : \abs{z} \leq r_* \}$ if 
$r_* < \infty$. In particular, 
if the set of zeros is unbounded, then $r_* = \infty$. If $P$ has no zero, we 
set $r_* = 0$.

The next theorem gives three bounds on $r_*$ depending on the coefficients of 
$P$, and thus yields inclusion disks for the zeros of $P$.

\begin{theorem} \label{thm:location}
Let a polyanalytic polynomial
\begin{equation}
P(z, \conj{z}) = \sum_{k=0}^n \sum_{j=0}^k \alpha_{j, k-j} z^j \conj{z}^{k-j}
\end{equation}
of degree $n$ be given such that there exists $\ell \in \{ 0, 1, \ldots, n \}$ 
with
\begin{equation} \label{eqn:dominant_term}
\alpha_n = \abs{\alpha_{\ell, n-\ell}} - \sum_{j=0, j \neq \ell}^n 
\abs{\alpha_{j,n-j}} > 0.
\end{equation}
Define the auxiliary real polynomial
\begin{equation}
q(t) = t^n - \sum_{k=0}^{n-1} c_k t^k
\quad \text{with} \quad
c_k = \sum_{j=0}^{k} \frac{\abs{\alpha_{j, k-j}}}{\alpha_n},
\quad k = 0, 1, \ldots, n-1.
\end{equation}
If not all coefficients $c_k$ are zero, then $q$ has a unique positive zero 
$r_0 > 0$, otherwise let $r_0 = 0$.  Then the following holds.
\begin{enumerate}
\item \label{it:r0_location}
All zeros of $P$ are located in the disk $\{ z \in \C : \abs{z} \leq r_0 \}$, 
i.e., $r_* \leq r_0$.

\item \label{it:r0_is_smaller}
The radius $r_0$ is bounded in terms of the coefficients of $P$ by
\begin{equation} \label{eqn:lagrange}
r_0 \leq r_1 = \max \left\{ 1, \sum_{k=0}^{n-1} \sum_{j=0}^k 
\frac{\abs{\alpha_{j, k-j}}}{\alpha_n} \right\}
\end{equation}
and
\begin{equation} \label{eqn:cauchy}
r_0 < r_2 = 1 + \max \left\{ \sum_{j=0}^k \frac{\abs{\alpha_{j, 
k-j}}}{\alpha_n} : k = 0, 1, \ldots, n-1 \right\}.
\end{equation}
\end{enumerate}
\end{theorem}

\begin{proof}
\ref{it:r0_location}
If $c_k = 0$ for $k = 0, 1, \ldots, n-1$, then $P$ has the form $P(z, \conj{z}) 
= \sum_{j=0}^n \alpha_{j,n-j} z^j \conj{z}^{n-j}$.  In that case, 
$\abs{P(z, \conj{z})} \geq \alpha_n \abs{z}^n$ shows that $P$ has only the zero 
$z = 0$, hence $r_* = r_0 = 0$.

If not all $c_k$ are zero, then there is exactly one sign change in the 
sequence of coefficients of 
$q$.  Thus, $q$ has exactly one positive zero $r_0$ by Descartes' rule of 
signs, see~\cite[Thm.~6.2d]{Henrici1974} or~\cite[p.~317]{Sheil-Small2002}, 
which is simple.
In particular, $q(t) > 0$ for $t > r_0$ and $q(t) < 0$ for $0 < t < r_0$.
Then, for $\abs{z} > r_0$, we obtain
\begin{align*}
\abs{P(z, \conj{z})}
&\geq \abs{\alpha_{\ell, n-\ell}} \abs{z}^n - \sum_{j=0, j \neq \ell}^n 
\abs{\alpha_{j,n-j}} \abs{z}^n - \sum_{k=0}^{n-1} \sum_{j=0}^k 
\abs{\alpha_{j, k-j}} \abs{z}^k \\
&= \alpha_n \abs{z}^n - \sum_{k=0}^{n-1} \sum_{j=0}^k 
\abs{\alpha_{j, k-j}} \abs{z}^k = \alpha_n q(\abs{z}) > 0.
\end{align*}
This shows that all zeros of $P$ are located in $\{ z \in \C : \abs{z} \leq r_0 
\}$, i.e., $r_* \leq r_0$.

\ref{it:r0_is_smaller} First, note that $r_1 \geq 1$ and $r_2 \geq 1$.  
Moreover, we have
\begin{equation*}
q(r_1)
= r_1^n - \sum_{k=0}^{n-1} c_k r_1^k
\geq r_1^n - \sum_{k=0}^{n-1} c_k r_1^{n-1}
\geq r_1^n - r_1 r_1^{n-1} = 0,
\end{equation*}
which implies $r_0 \leq r_1$.  Similarly, with $c_k \leq r_2 - 1$ for all $k$, 
we have
\begin{equation*}
q(r_2) = r_2^n - \sum_{k=0}^{n-1} c_k r_2^k
\geq r_2^n - \sum_{k=0}^{n-1} (r_2 - 1) r_2^k
= r_2^n - (r_2^n - 1) = 1 > 0,
\end{equation*}
which implies $r_0 < r_2$.
\end{proof}

\begin{remark}
\begin{enumerate}
\item
The bound $r_0$ on the maximum modulus $r_*$ of the zeros of $P$ generalizes a 
classical result of Cauchy for analytic polynomials;
see~\cite[p.~122]{Cauchy1829}, \cite[Thm.~6.4l]{Henrici1974}, 
or~\cite[Thm.~27.1]{Marden1966}.

\item
For an analytic polynomial $p(z) = \sum_{k=0}^n a_k z^k$ of degree~$n$,
\eqref{eqn:lagrange} becomes Lagrange's bound
\begin{equation}
r_1 = \max \biggl\{ 1, \sum_{k=0}^{n-1} \frac{\abs{a_k}}{\abs{a_n}} \biggr\},
\end{equation}
and~\eqref{eqn:cauchy} becomes Cauchy's bound
\begin{equation}
r_2 = 1 + \max \biggl\{ \frac{\abs{a_k}}{\abs{a_n}} : k = 0, 1, \ldots, n-1 
\biggr\};
\end{equation}
see~\cite[Thm.~27.2]{Marden1966} or Cauchy's original 
work~\cite[p.~122]{Cauchy1829}.
Theorem~\ref{thm:location} generalizes these bounds to polyanalytic 
polynomials, and therefore we call~\eqref{eqn:lagrange} the \emph{Lagrange 
bound} and~\eqref{eqn:cauchy} the \emph{Cauchy bound} on the zeros of a 
polyanalytic polynomial.
Moreover, Theorem~\ref{thm:location} generalizes a bound for harmonic 
polynomials by Sheil-Small~\cite[Sect.~2.6.10]{Sheil-Small2002}.

\item
For a polyanalytic polynomial of degree $n$ that does not 
satisfy~\eqref{eqn:dominant_term}, the 
set of zeros may be unbounded.  In this case, there is no finite bound on the 
modulus of the zeros, e.g., $P(z, \conj{z}) = z^n + \conj{z}^n$, $n \geq 1$, 
has an unbounded set of nonisolated zeros.
\end{enumerate}
\end{remark}

Theorem~\ref{thm:location} does not require that the number of zeros of $P$ is 
finite, e.g., $P(z, \conj{z}) = z \conj{z} - 1$ vanishes on the unit circle, 
satisfies $r_* = 1$, and the bounds are $r_0 = r_1 = 1$ and $r_2 = 2$.

The bound $r_0$ and Lagrange's bound~$r_1$ are sharp in the sense that for 
all $n$ there exists a polyanalytic polynomial $P$ of degree~$n$ with $r_* = 
r_0 = r_1$; see Example~\ref{ex:lagrange_bound}.
Cauchy's bound~\eqref{eqn:cauchy} is not sharp since $r_* \leq r_0 < r_2$, 
but asymptotically sharp in the sense that the relative error $(r_2 - r_*)/r_*$ 
can be arbitrarily small for any $n$; see Example~\ref{ex:cauchy_bound}.

\begin{example} \label{ex:lagrange_bound}
We consider
\begin{equation}
P_1(z, \conj{z}) = a z^n + b \conj{z}
\end{equation}
with $n \geq 2$ and $a, b \in \C \setminus \{ 0 \}$.
First, we determine the zeros of $P_1$.  Clearly, $z = 0$ is a zero.  For the 
other zeros, write $z = \abs{z} \ee^{\ii \varphi}$, $a = \abs{a} \ee^{\ii 
\alpha}$, and $b = \abs{b} \ee^{\ii \beta}$ with $\varphi, \alpha, \beta \in 
\R$.
Then $P_1(z, \conj{z}) = 0$ is equivalent to $\abs{z}^{n-1} \ee^{\ii (n+1) 
\varphi} = \frac{\abs{b}}{\abs{a}} \ee^{\ii (\beta - \alpha + \pi)}$.  Hence, 
the zeros of $P_1$ are
\begin{equation*}
z_0 = 0, \quad
z_j = \biggl( \frac{\abs{b}}{\abs{a}} \biggr)^{\frac{1}{n-1}} \ee^{\ii 
\varphi_j}, 
\quad \varphi_j = \frac{\beta - \alpha + (2 j + 1) \pi}{n+1}, \quad j = 1, 
\ldots, n+1.
\end{equation*}
Thus, the largest magnitude of a zero of $P_1$ is $r_* = 
(\abs{b}/\abs{a})^{1/(n-1)}$.  The Lagrange and Cauchy bounds are $r_1 = 
\max \{ 1, \abs{b}/\abs{a} \}$ and $r_2 = 1 + \frac{\abs{b}}{\abs{a}}$, 
respectively.
The auxiliary polynomial $q(t) = t^n - \frac{\abs{b}}{\abs{a}} t$ has the 
positive zero $r_0 = (\abs{b}/\abs{a})^{1/(n-1)} = r_*$.
If $\abs{a} = \abs{b}$, then $r_* = r_0 = r_1 < r_2$, which shows that the 
bounds $r_0$ and $r_1$ are sharp.
\end{example}

\begin{example} \label{ex:cauchy_bound}
We consider
\begin{equation}
P_2(z, \conj{z}) = a z^n + b \conj{z}^{n-1}
\end{equation}
with $n \geq 2$ and $a, b \in \C \setminus \{ 0 \}$.
Write $a = \abs{a} \ee^{\ii \alpha}$, $b = \abs{b} \ee^{\ii \beta}$ with 
$\alpha, \beta \in \R$.
The zeros of $P_2$, computed as for $P_1$, are
\begin{equation*}
z_0 = 0, \quad
z_j = \frac{\abs{b}}{\abs{a}} \ee^{\ii \varphi_j}, \quad
\varphi_j = \frac{\beta - \alpha + (2 j + 1) \pi}{2n-1}, \quad j = 1, 
\ldots, 2n-1.
\end{equation*}
Thus, the optimal radius is $r_* = \abs{b}/\abs{a}$.  The auxiliary polynomial 
$q(t) = t^n - \frac{\abs{b}}{\abs{a}} t^{n-1}$ has the unique positive zero 
$r_0 = \abs{b}/\abs{a} = r_*$.
Moreover, $r_1 = \max \{ 1, \abs{b}/\abs{a} \}$ and $r_2 = 1 + 
\abs{b}/\abs{a}$. 
The relative error $(r_2 - r_*)/r_* = \abs{a}/\abs{b}$ can be arbitrarily small.
\end{example}

\begin{figure}
{\centering
\includegraphics[width=0.48\linewidth]{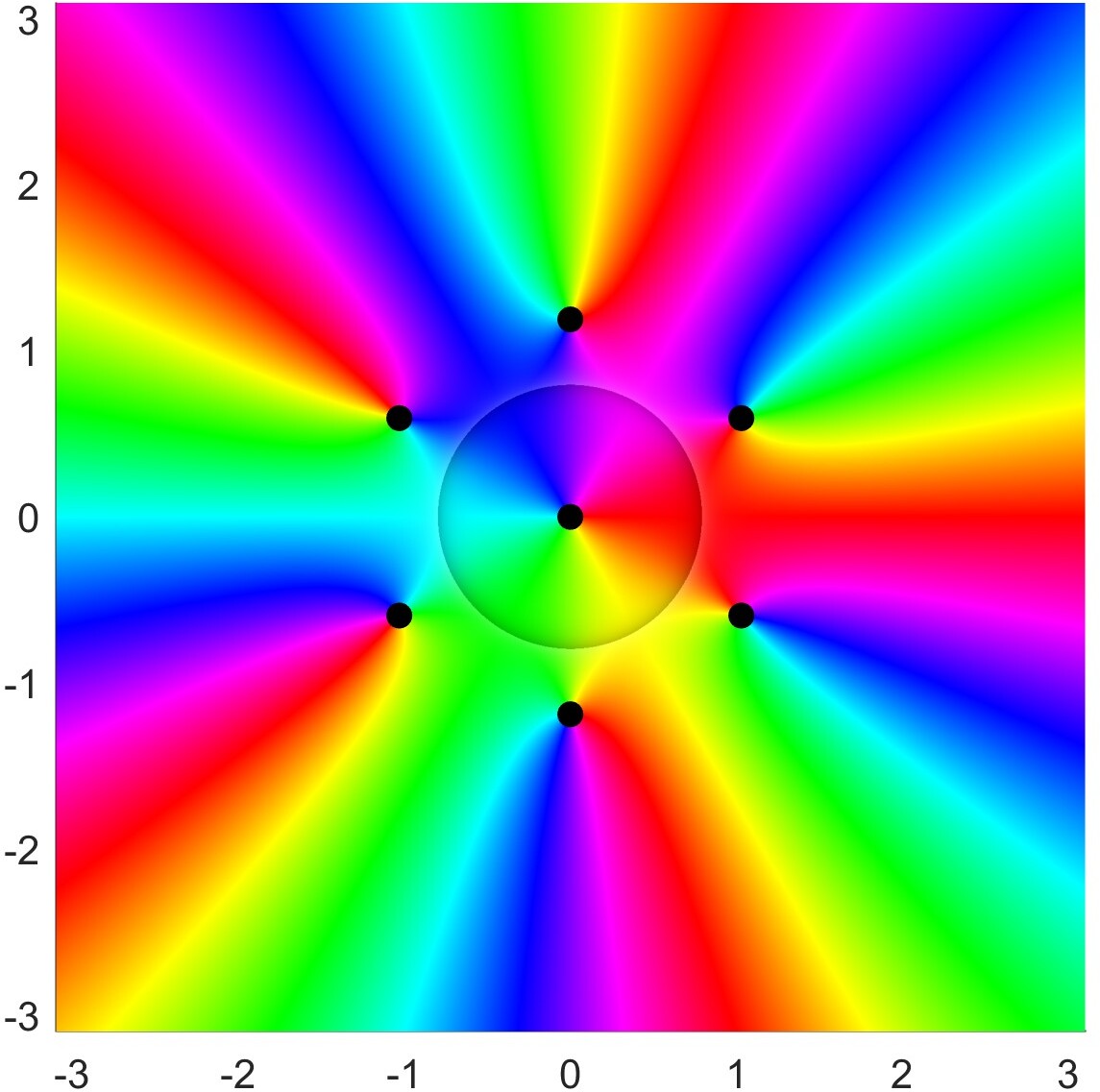}
\includegraphics[width=0.48\linewidth]{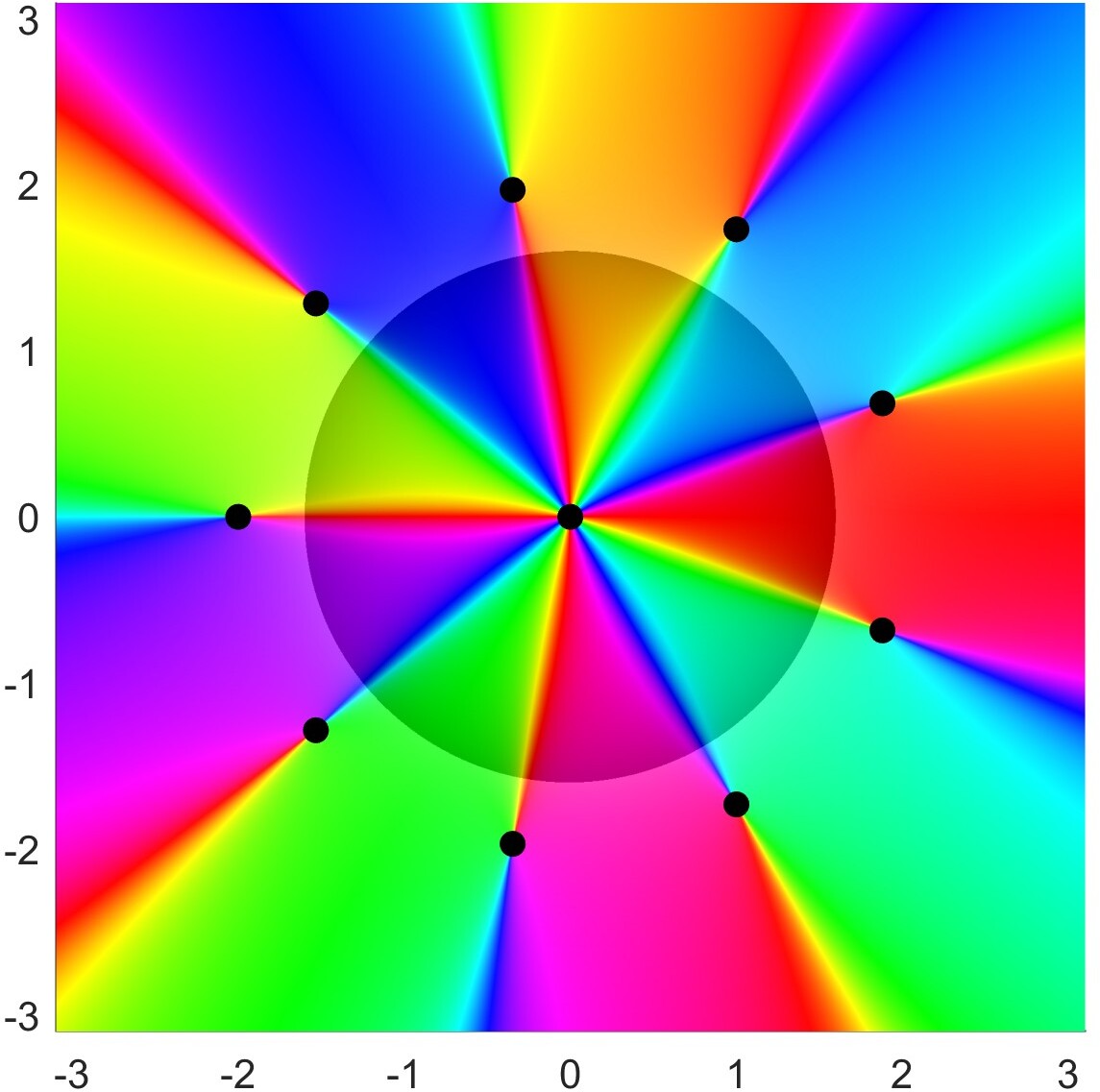}

}
\caption{Phase plots of $P_1$ (left) and $P_2$ (right) in 
Examples~\ref{ex:lagrange_bound} and~\ref{ex:cauchy_bound} with $a = 1$, $b = 
2$, and $n = 5$.}
\label{fig:location_of_zeros}
\end{figure}

We visualize the polyanalytic polynomials $P_1$ and $P_2$ with a \emph{phase 
plot}, in which the domain of a complex function $f$ is colored according to 
its \emph{phase} $f(z)/\abs{f(z)}$; see~\cite{Wegert2012,WegertSemmler2011}.  
Here, we use a phase plot with a custom shading indicating the orientation of 
$f$.
Recall that $f$ is sense-preserving (orientation-preserving) at $z$ if $J_f(z) 
> 0$, and sense-reversing (orientation-reversing) at $z$ if $J_f(z) < 0$,
where $J_f = \abs{\partial_z f}^2 - \abs{\partial_{\conj{z}} f}^2$ is the
Jacobian of $f$; see, e.g.,~\cite[p.~45]{Balk1991} or~\cite[p.~5]{Duren2004}.
In the phase plots, $f$ is sense-preserving in the brighter regions and 
sense-reversing in the darker regions.
Moreover, black dots mark the zeros of the functions.
Figure~\ref{fig:location_of_zeros} displays phase plots of $P_1$ (left) and 
$P_2$ (right).
The polyanalytic polynomial $P_1$ is sense-reversing at its zero $z=0$ and 
sense-preserving 
at the other zeros.  Note that the colors appear in the same order around all 
zeros at which $P_1$ is sense-preserving and that this order is reversed at 
the zero $z = 0$ where $P_1$ is sense-reversing.
The phase of $P_2$ behaves similarly, except at the origin where $P_2$ is 
singular (i.e., $J_{P_2}(0) = 0$) if $n \geq 3$.
We observe that all colors appear $n-1 = 4$ times around $z=0$,
indicating that the index of $P_2$ at the origin is $-(n-1)$;
see~\cite{Balk1991} for more on the indices of continuous functions 
and~\cite{SeteZur2021, SuffridgeThompson2000} for the indices of harmonic 
functions.

\begin{figure}
{\centering
\includegraphics[width=0.48\linewidth]{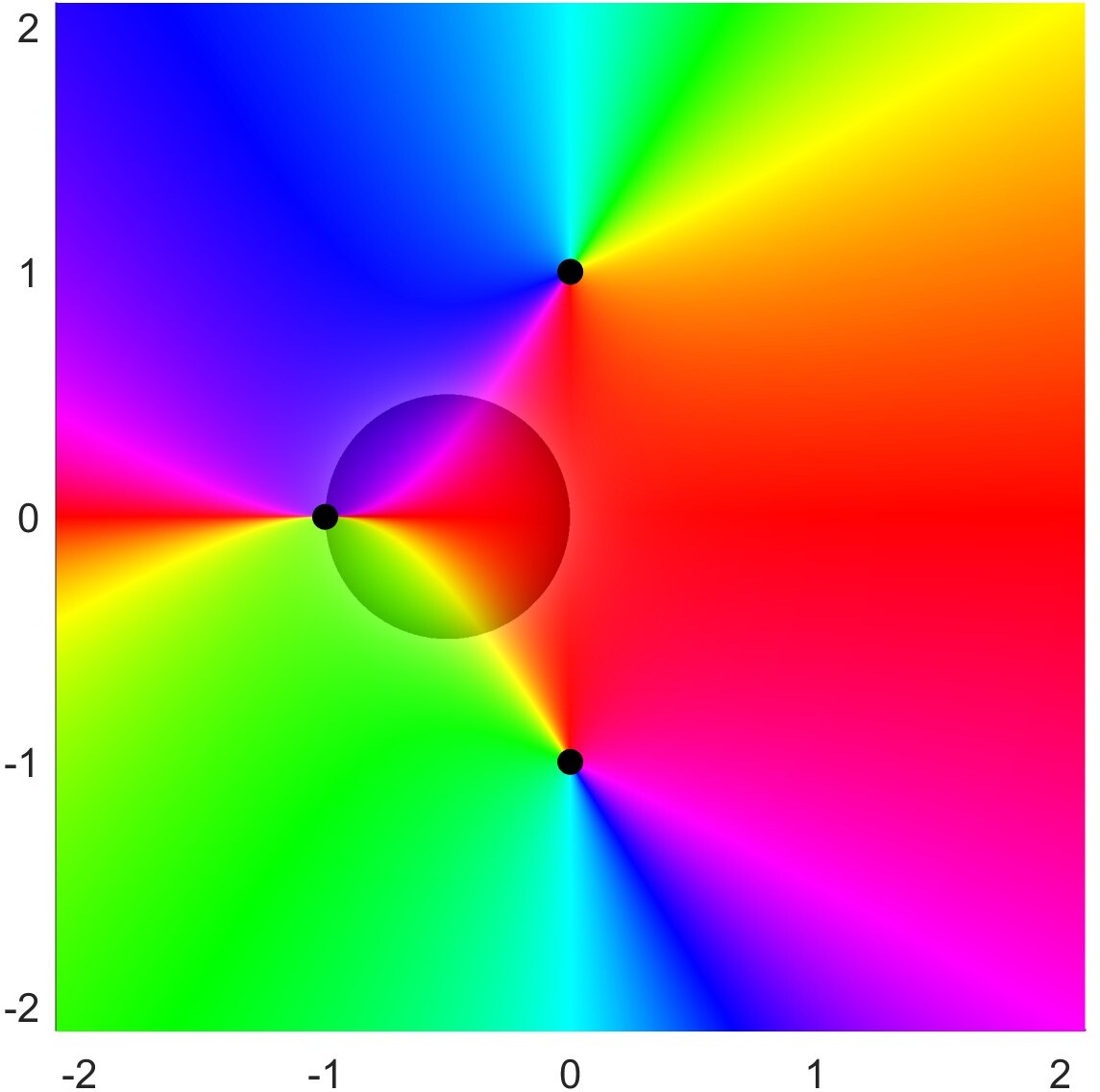}
\includegraphics[width=0.48\linewidth]{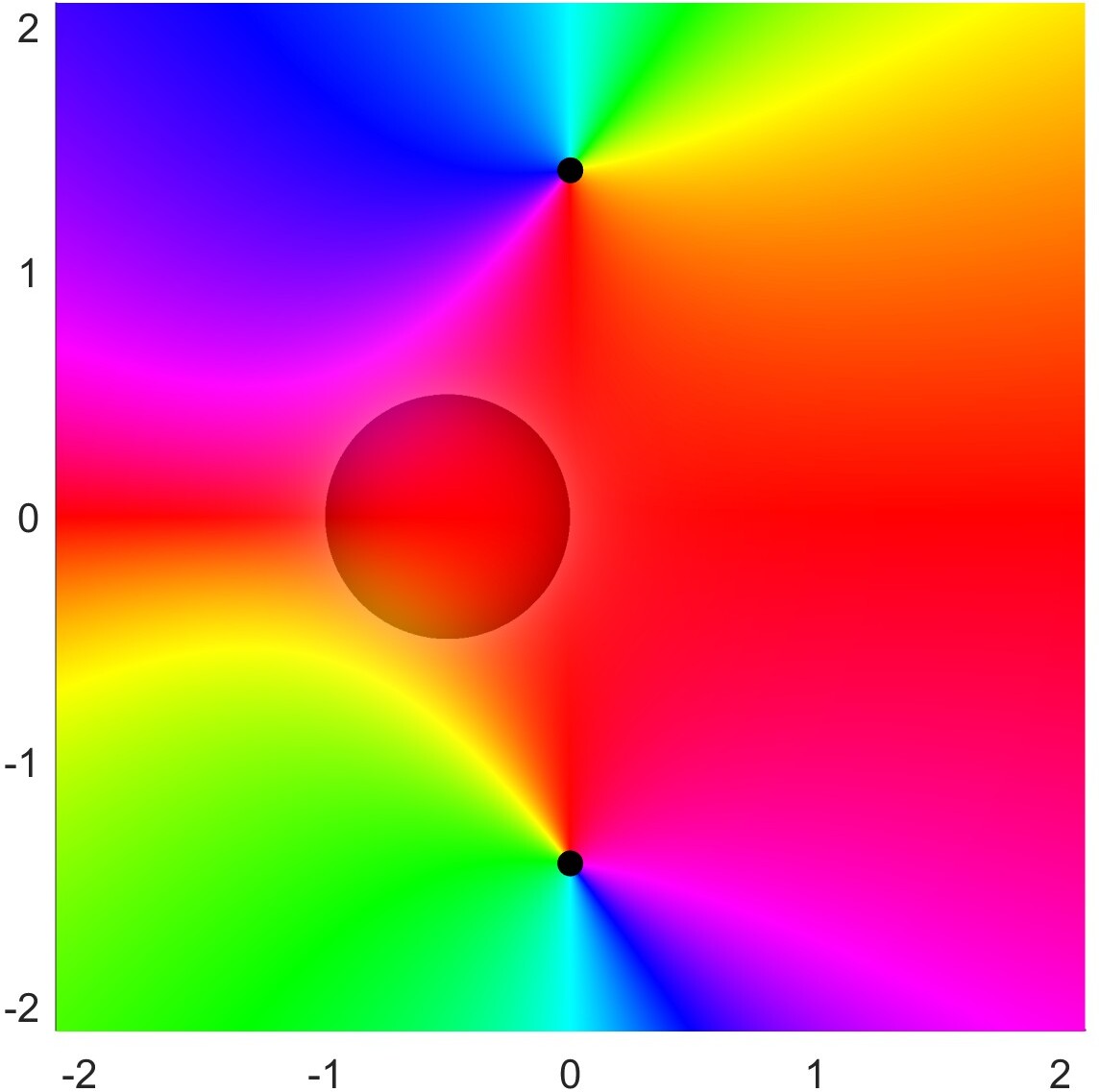}

}
\caption{Phase plots of $P_3$ (left) and $P_4$ (right) from 
Examples~\ref{ex:P3} and~\ref{ex:P4}.}
\label{fig:P34}
\end{figure}

In the previous examples, $r_0$ is an optimal bound on the magnitude of the 
zeros, but this is not always the case.

\begin{example} \label{ex:P3}
The polyanalytic polynomial
\begin{equation}
P_3(z, \conj{z}) = z^2 + z + \conj{z} + 1
\end{equation}
has the three zeros $-1$, $\ii$, and $- \ii$ (with indices $0$, $1$, and $1$, 
respectively), hence $r_* = 1$.
Here, $r_1 = \max \{ 1, 3 \} = 3$, $r_2 = 1 + \max \{ 1, 2 \} = 3$.
The auxiliary polynomial $q(t) = t^2 - 2 t - 1$ has the positive zero $r_0 = 1 
+ \sqrt{2}$.  Thus, we have $r_* < r_0 < r_1 = r_2$ in this example.
\end{example}

In all examples above, Lagrange's bound is smaller than Cauchy's bound, 
i.e., $r_1 \leq r_2$, which does not hold in general.

\begin{example} \label{ex:P4}
The Lagrange and Cauchy bounds for the polyanalytic polynomial
\begin{equation}
P_4(z, \conj{z}) = z^2 + z + \conj{z} + 2
\end{equation}
are $r_1 = 4 > 3 = r_2$.
The zeros of $P_4$ are $\pm \sqrt{2} \ii$ (each with index $1$), hence the 
optimal radius is $r_* = \sqrt{2}$.
\end{example}

\begin{table}
\begin{center}
\bgroup
\renewcommand{\arraystretch}{1.5}
\begin{tabular}{lcccc}
\toprule
Function & $r_*$ & $r_0$ & $r_1$ & $r_2$ \\
\midrule
$P_1(z, \conj{z}) = a z^n + b \conj{z}$ & 
$\Bigl( \frac{\abs{b}}{\abs{a}} \Bigr)^{\frac{1}{n-1}}$ & 
$\Bigl( \frac{\abs{b}}{\abs{a}} \Bigr)^{\frac{1}{n-1}}$ &
$\max \Bigl\{ 1, \frac{\abs{b}}{\abs{a}} \Bigr\}$ &
$1 + \frac{\abs{b}}{\abs{a}}$ \\

$P_2(z, \conj{z}) = a z^n + b \conj{z}^{n-1}$ & $\frac{\abs{b}}{\abs{a}}$ & 
$\frac{\abs{b}}{\abs{a}}$ & 
$\max \Bigl\{ 1, \frac{\abs{b}}{\abs{a}} \Bigr\}$ & $1 + 
\frac{\abs{b}}{\abs{a}}$ \\

$P_3(z, \conj{z}) = z^2 + z + \conj{z} + 1$ & $1$ & $1 + \sqrt{2}$ & $3$ & $3$ 
\\

$P_4(z, \conj{z}) = z^2 + z + \conj{z} + 2$ & $\sqrt{2}$ & $1 + \sqrt{3}$ & $4$ 
& $3$ \\
\bottomrule
\end{tabular}
\egroup
\end{center}
\caption{Maximum magnitude $r_*$ of the zeros and its bounds $r_0, r_1, r_2$ 
for $P_1, \ldots, P_4$ in Examples~\ref{ex:lagrange_bound}--\ref{ex:P4}.}
\label{tab:compare_radii}
\end{table}

Figure~\ref{fig:P34} shows phase plots of $P_3$ and $P_4$ with their zeros.
Table~\ref{tab:compare_radii} compares the optimal radius $r_*$ and the radii 
$r_0, r_1, r_2$ for $P_1, \ldots, P_4$.
 \section{Constructing extremal polyanalytic polynomials}\label{sect:extremal}

This section is devoted to construct an \emph{extremal} polyanalytic polynomial
of degree $n$, i.e., with the maximum number of $n^2$ zeros; see 
Corollary~\ref{cor:wilmshurst}.
Moreover, we approximately locate its zeros.
Previously, Bshouty, Hengartner, and Suez~\cite{BshoutyHengartnerSuez1995} and
Wilmshurst~\cite[p.~2080]{Wilmshurst1998} gave examples of extremal harmonic 
and therefore polyanalytic polynomials.
Our construction is similar to the one in~\cite{BshoutyHengartnerSuez1995}, but 
significantly simplified.

We iteratively define the harmonic polynomials $P_1(z, \conj{z}) = a_1 
\conj{z}$ and
\begin{equation}
P_n(z, \conj{z}) = 
\begin{cases} 
a_n z^n + P_{n-1}(z, \conj{z}), & \text{if } n \text{ is even},\\
a_n \conj{z}^n + P_{n-1}(z, \conj{z}), & \text{if } n \text{ is odd},
\end{cases}
\quad n \geq 2,
\end{equation}
with appropriate coefficients $a_1, \ldots, a_n$.
Our proof that $P_n$ has $n^2$ zeros relies on the argument principle for 
continuous functions, and on results about the indices of harmonic functions at 
their zeros.
Given $P_{n-1}$ with $(n-1)^2$ zeros, we show that $P_n$ has one zero close to 
each of the zeros of $P_{n-1}$, and that $P_n$ has $2n-1$ additional zeros in 
a certain annulus surrounding the zeros of $P_{n-1}$, provided that 
$\abs{a_n}$ is sufficiently small.
The construction is illustrated in Figure~\ref{fig:extremal_construction}.

\begin{theorem} \label{thm:extremal_number}
For all $n \in \N$, $n \geq 1$, there exist $a_1, \dots, a_n \in \C \setminus 
\{ 0 \}$, such that the harmonic polynomial
\begin{equation}\label{eqn:extremal_polynomial}
P_n(z, \conj{z}) = \sum_{k = 1}^{\floor{\frac{n}{2}}} a_{2k}z^{2k} + \sum_{k 
= 1}^{\ceil{\frac{n}{2}}} a_{2k-1}\conj z^{2k-1}
\end{equation}
has exactly $n^2$ distinct zeros.
\end{theorem}

\begin{proof}
The harmonic polynomial $P_n$ in~\eqref{eqn:extremal_polynomial} has at most 
$n^2$ zeros by Corollary~\ref{cor:wilmshurst}.
We show inductively that $P_n$ has $n^2$ zeros $z_1, \ldots, z_{n^2}$ and that 
the index of $P_n$ at each zero is either $1$ or $-1$, i.e., $\ind(P_n; z_j) 
= \pm 1$.

For $n = 1$ and any $a_1 \neq 0$, the harmonic polynomial $P_1(z, \conj{z}) = 
a_1 \conj{z}$ has the unique zero $z_1 = 0$.
The Jacobian of $P_1$ is $J_{P_1}(z) = - a_1^2 < 0$, hence $P_1$ is 
sense-reversing at $z_1$ and $\ind(P_1; z_1) = -1$; 
see~\cite[Prop.~2.7]{SeteZur2021}, \cite[p.~413]{DurenHengartnerLaugesen1996}, 
or~\cite[p.~66]{SuffridgeThompson2000}.

For $n = 2$ and any $a_1, a_2 \neq 0$, the polynomial $P_2(z, \conj{z}) = 
a_2 z^2 + a_1 \conj{z}$ has the four zeros $z_1 = 0$ and
$z_j = \frac{\abs{a_1}}{\abs{a_2}} \ee^{\ii \varphi_j}$ with $\varphi_j = 
(\arg(a_1/a_2) +  (2j+1)\pi)/3$, $j = 2, 3, 4$; see 
Example~\ref{ex:cauchy_bound}.
The Jacobian of $P_2$ is $J_{P_2}(z) = \abs{2 a_2 z}^2 - \abs{a_1}^2$,
hence $J_{P_2}(z_1) < 0$ and $J_{P_2}(z_j) > 0$ for $j = 2, 3, 4$.
Therefore, $\ind(P_2; z_1) = -1$ and $\ind(P_2; z_j) = +1$ for $j = 2, 3, 4$.

For the induction step, assume for $n-1 \geq 2$ that $P_{n-1}$ 
has $(n-1)^2$ zeros $z_1, \ldots, z_{(n-1)^2}$ with $\ind(P_{n-1}; 
z_j) = \pm 1$ for $j = 1, \ldots, (n-1)^2$.
If $n$ is even, then $P_{n-1}$ and $P_n$ have the form $P_{n-1}(z) = a_{n-1} 
\conj{z}^{n-1} + P_{n-2}(z)$ and
\begin{equation*}
P_n(z, \conj{z}) = a_n z^n + P_{n-1}(z, \conj{z}).
\end{equation*}

Since $\partial_z P_n(z, \conj{z}) = n a_n z^{n-1} + \partial_z P_{n-1}(z, 
\conj{z})$ depends on $a_n$ while $\partial_{\conj{z}} P_n$ does not, 
it is possible to choose $a_n$ such that 
$\partial_z P_n$ and $\partial_{\conj{z}} P_n$ have no common zero.
Then the index of $P_n$ at a zero $z$ satisfies $\ind(P_n; z) \in \{ -1, 0, 1 
\}$; see~\cite[p.~67]{Sheil-Small2002}.

Next, we show that $P_n$ has $n^2$ zeros and that the index of $P_n$ at the 
zeros is $\pm 1$ for sufficiently small $\abs{a_n} \neq 0$ in two steps:
(i) $P_n$ has a zero close to each of the $(n-1)^2$ zeros of $P_{n-1}$, and
(ii) $P_n$ has $2n-1$ additional zeros.
In the following, let $D_\delta(a) = \{ z \in \C : \abs{z - a} < \delta \}$ 
denote the open disk with radius $\delta > 0$ and center $a \in \C$.

(i)
Using Rouch\'e's theorem (see~\cite[p.~37]{Balk1991}
or~\cite[Thm.~2.3]{SeteLuceLiesen2015}), we prove that $P_n(z, \conj{z}) = a_n 
z^n + 
P_{n-1}(z, \conj{z})$ has a zero close to each zero of $P_{n-1}$
for sufficiently small $\abs{a_n} \neq 0$.
Let $\delta > 0$ be such that $D_\delta(z_j) \cap D_\delta(z_k) = \emptyset$ 
for $j, k = 1, \ldots, (n-1)^2$, $j \neq k$.
Note that $0 \notin \partial D_\delta(z_j)$ for $j = 1, \ldots, (n-1)^2$, 
since $0$ is one of the zeros of $P_{n-1}$.
Define
\begin{equation*}
m = \min_{j = 1, \ldots, (n-1)^2} \min_{z \in \partial D_\delta(z_j)} 
\frac{\abs{P_{n-1}(z, \conj{z})}}{\abs{z^n}} > 0.
\end{equation*}
For $0 < \abs{a_n} < m$ and for each $j = 1, \ldots, (n-1)^2$, we have
\begin{equation*}
\abs{P_n(z, \conj{z}) - P_{n-1}(z, \conj{z})}
= \abs{a_n z^n} < \abs{P_{n-1}(z, \conj{z})}
\quad \text{for } z \in \partial D_\delta(z_j).
\end{equation*}
Hence, $P_n$ and $P_{n-1}$ have the same winding on $\partial D_\delta(z_j)$ by 
Rouch\'e's theorem.
By the argument principle for harmonic functions 
(see, e.g.,~\cite[Thm.~2.5]{SeteZur2021}),
\begin{equation*}
\wind(P_{n-1}; \partial D_\delta(z_j)) = \ind(P_{n-1}; z_j) = \pm 1.
\end{equation*}
Next, we use the argument principle for $P_n$.  Since $\abs{\ind(P_n; z)} 
\leq 1$ holds at a zero of $P_n$, there exists at least one zero of $P_n$ with 
index 
$\ind(P_{n-1}; z_j) = \pm 1$ in each disk $D_\delta(z_j)$.

(ii)
We show that $P_n$ has $2n-1$ additional zeros.
Let $r_{n-1} > 0$ such that all zeros of $P_{n-1}$ are in the interior of the 
circle $ C_{n-1} = \{ z \in \C : \abs{z} = r_{n-1} \}$.
Then the winding of $P_{n-1}$ along $C_{n-1}$, oriented in the positive sense, 
is $\wind(P_{n-1}; C_{n-1}) = -(n-1)$.
For $0 < \abs{a_n} < r_{n-1}^{-n} \min_{z \in C_{n-1}} \abs{P_{n-1}(z, 
\conj{z})}$, we have
\begin{equation*}
\abs{P_n(z, \conj{z}) - P_{n-1}(z, \conj{z})} = \abs{a_n z^n} < 
\abs{P_{n-1}(z, \conj{z})}
\quad \text{for } z \in C_{n-1}.
\end{equation*}
Hence, by Rouch\'e's theorem,
\begin{equation*}
\wind(P_n; C_{n-1}) = \wind(P_{n-1}; C_{n-1}) = -(n-1).
\end{equation*}
Let $r_n > r_{n-1}$ such that the zeros of $P_n$ satisfy $\abs{z} < r_n$ and 
let $C_n = \{ z \in \C : \abs{z} = r_n \}$.
Then $\wind(P_n; C_n) = n$ and the winding of $P_n$ along the boundary of 
the annulus $A = \{ z \in \C : r_{n-1} < \abs{z} < r_n \}$ is
\begin{equation*}
\wind(P_n; \partial A) = n - (-(n-1)) = 2n-1,
\end{equation*}
where $\partial A$ is oriented such that $A$ lies to the left.
By the argument principle, and since $\abs{\ind(P_n; z)} \leq 1$ at 
zeros of $P_n$, there are at least $2n-1$ zeros of $P_n$ in $A$ with $\ind(P_n; 
z) = +1$.

Combining (i) and (ii), $P_n$ has at least $(n-1)^2 + 2n - 1 = n^2$ zeros and, 
together with Corollary~\ref{cor:wilmshurst}, $P_n$ has exactly $n^2$ zeros.
More precisely, $P_n$ has exactly one zero in each disk $D_\delta(z_j)$
(each with index $\ind(P_{n-1}; z_j) = \pm 1$), and exactly $2n-1$ zeros in 
the annulus $A$ (each with index $+1$).

Finally, if $n$ is odd, then $P_{n-1}(z, \conj{z}) = a_{n-1} z^{n-1} + 
P_{n-2}(z, \conj{z})$ and the above reasoning applies to $\conj{P_{n-1}(z, 
\conj{z})}$, i.e., there exists $a_n \neq 0$ such that $a_n z^n + 
\conj{P_{n-1}(z, \conj{z})}$ has $n^2$ zeros.
Therefore, $P_n(z, \conj{z}) = \conj{a}_n \conj{z}^n + P_{n-1}(z, \conj{z})$ 
has 
$n^2$ zeros.  Note that, due to taking the complex conjugate, the indices of 
$P_n$ at its zeros in $A$ are $-1$ instead of $+1$.
This completes the proof.
\end{proof}

\begin{remark}
Truncating the lower order terms of $P_k$ yields the polynomial
\begin{equation}
Q_k(z, \conj{z}) = a_k z^k + a_{k-1} \conj{z}^{k-1}, \quad k = 2, \ldots, n,
\end{equation}
which has the zeros $\zeta_{1,1} = 0$ and
\begin{equation} \label{eqn:approx_zeros}
\zeta_{k,j} = \frac{\abs{a_{k-1}}}{\abs{a_k}} \exp \left(\ii 
\frac{\arg(a_{k-1}/a_k) + (2j+1)\pi}{2k-1} \right), \quad j = 1, \dots, 2k-1;
\end{equation}
see Example~\ref{ex:cauchy_bound}.  These are located on the circle $\abs{z} = 
\abs{a_{k-1}}/\abs{a_k}$ with equispaced angles.
If $\abs{a_1}, \abs{a_2}, \ldots, \abs{a_n}$ are sufficiently small, the zeros 
of $P_n$ are approximately located at the points $\zeta_{k,j}$.
Moreover, using the Newton-Kantorovich theorem, one can show that the Newton 
iteration with initial point $\zeta_{k,j}$ converges to a zero of $P_n$, and 
every zero of $P_n$ is obtained in this way.
We omit the details and refer to~\cite[Sect.~4.1]{SeteZur2020}.
\end{remark}

\begin{figure}[t]
{\centering
\includegraphics[width = .327\linewidth]{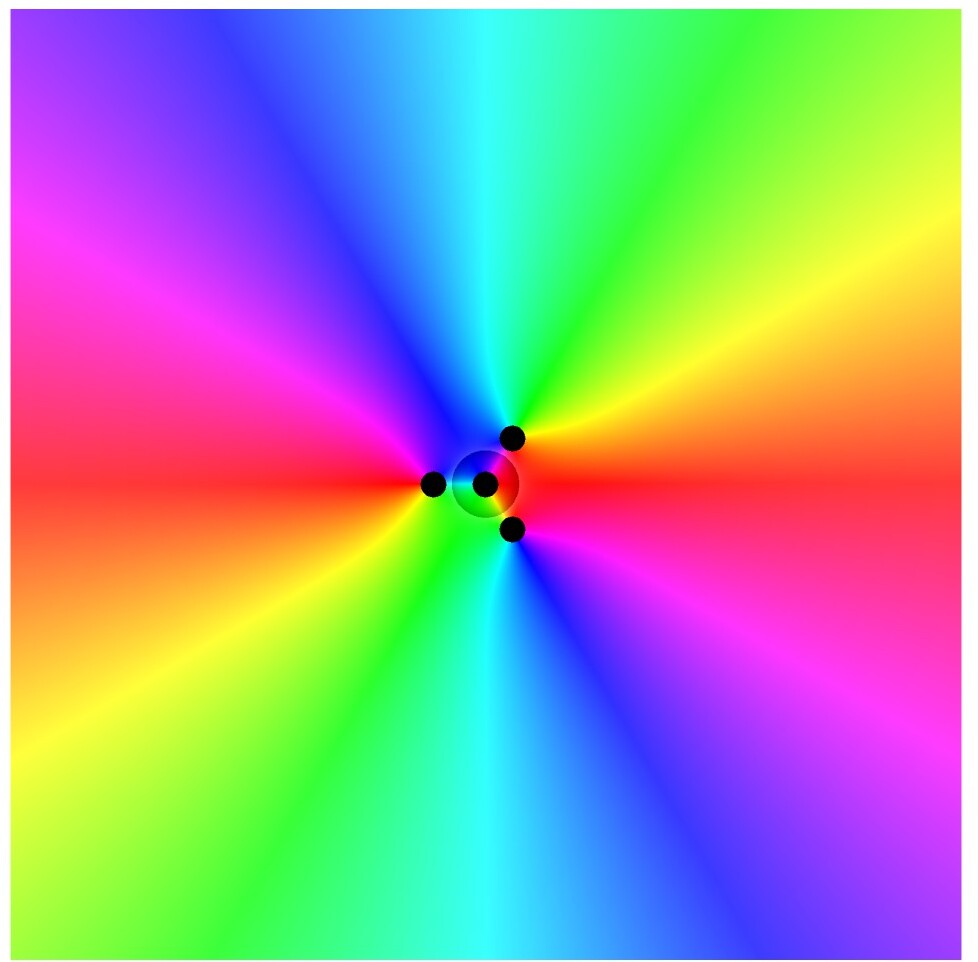}
\hfill
\includegraphics[width = .327\linewidth]{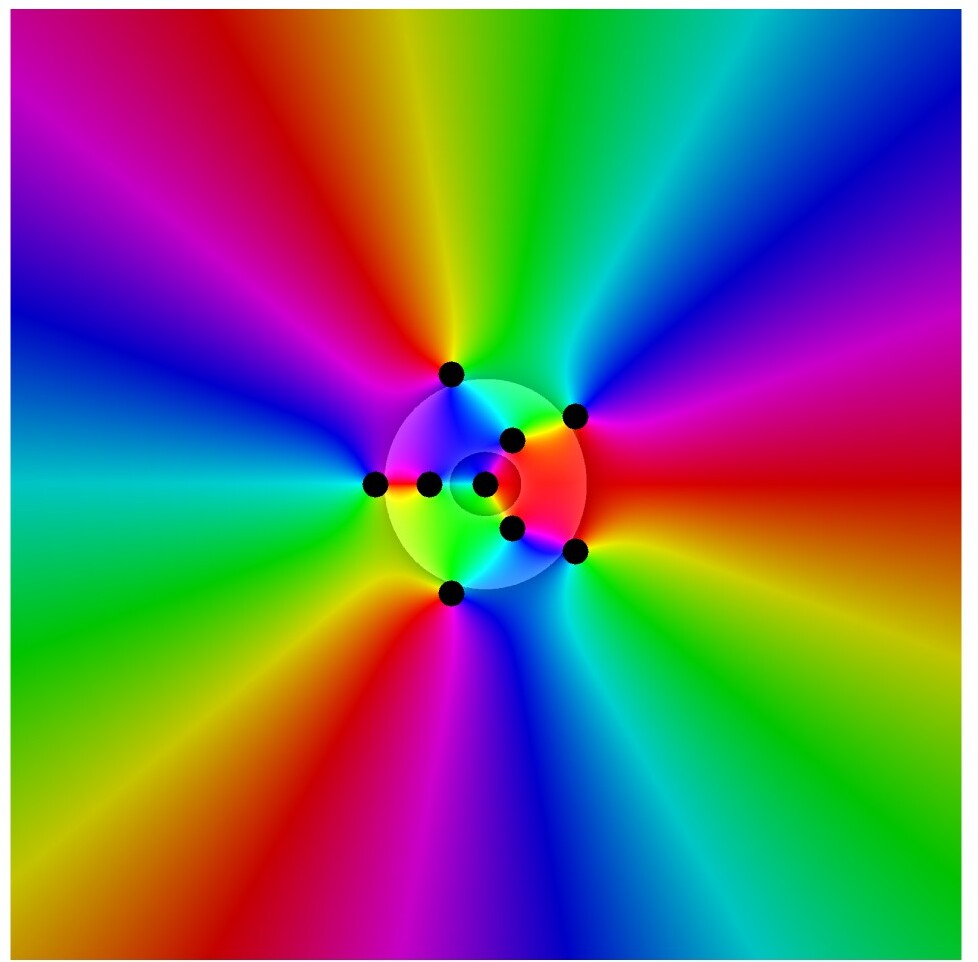}
\hfill
\includegraphics[width = .327\linewidth]{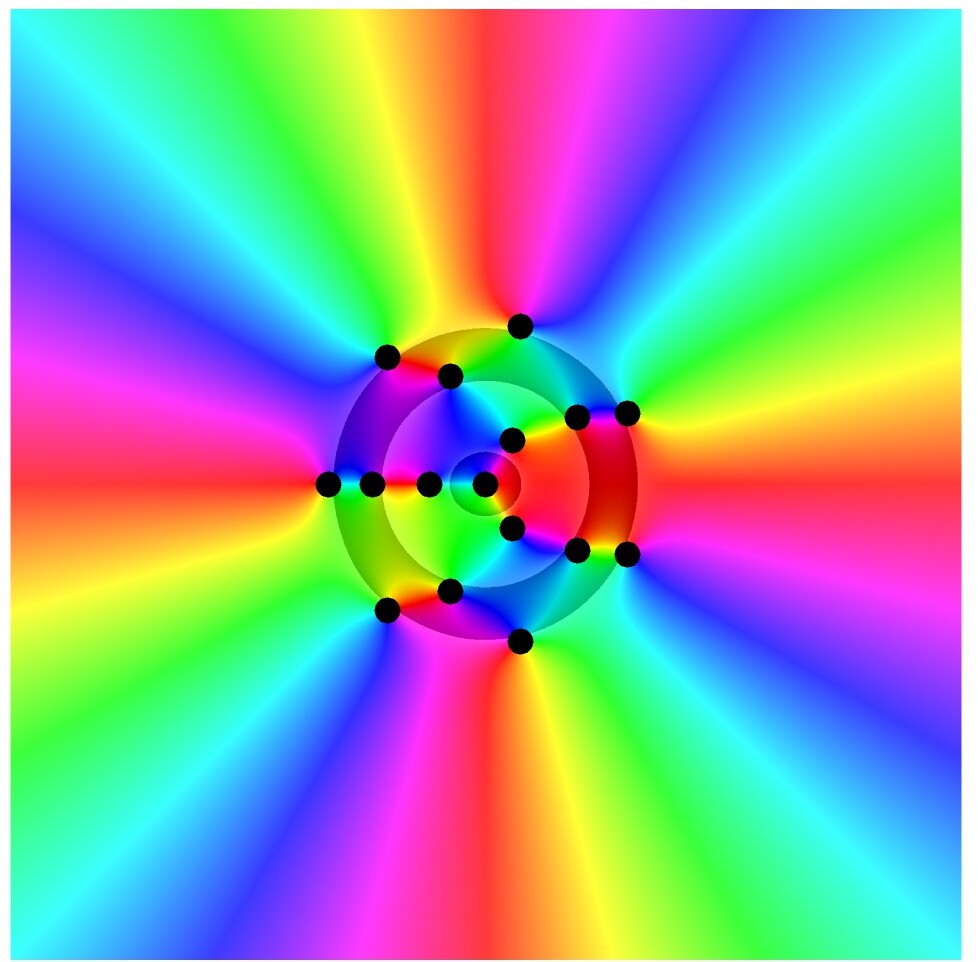}

\smallskip

\includegraphics[width = .327\linewidth]{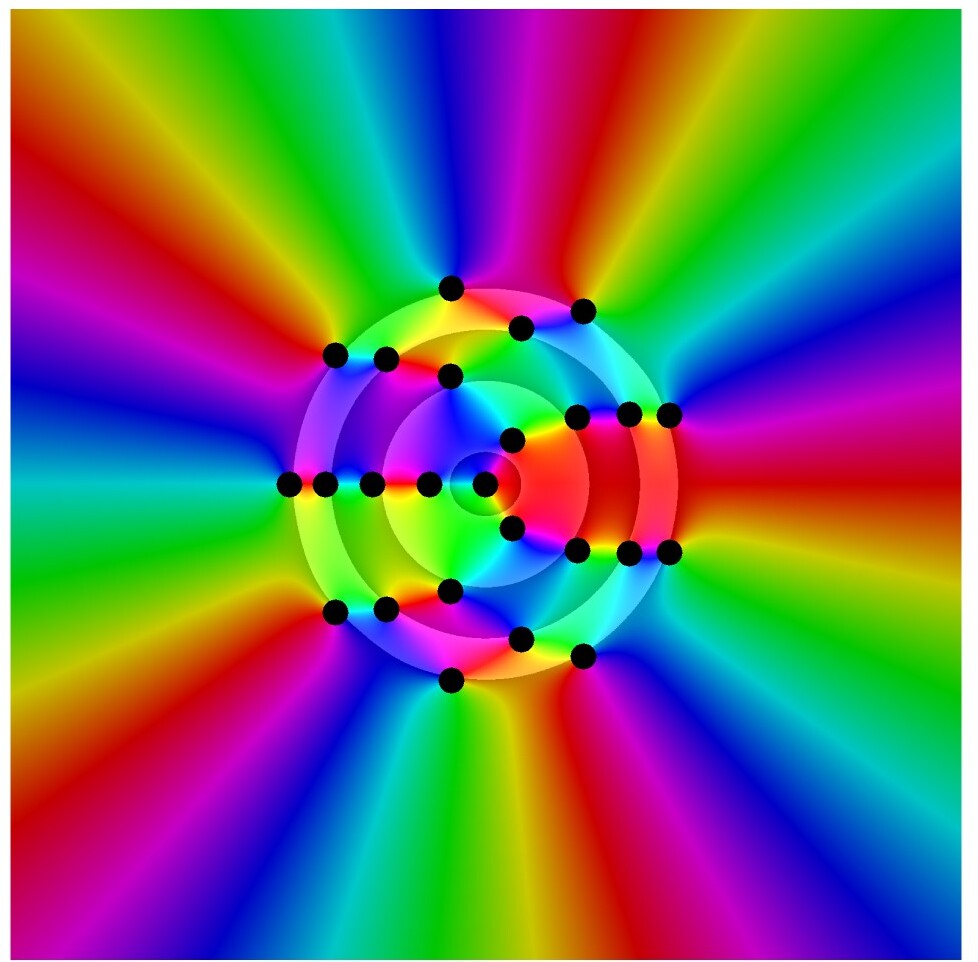}
\hfill
\includegraphics[width = .327\linewidth]{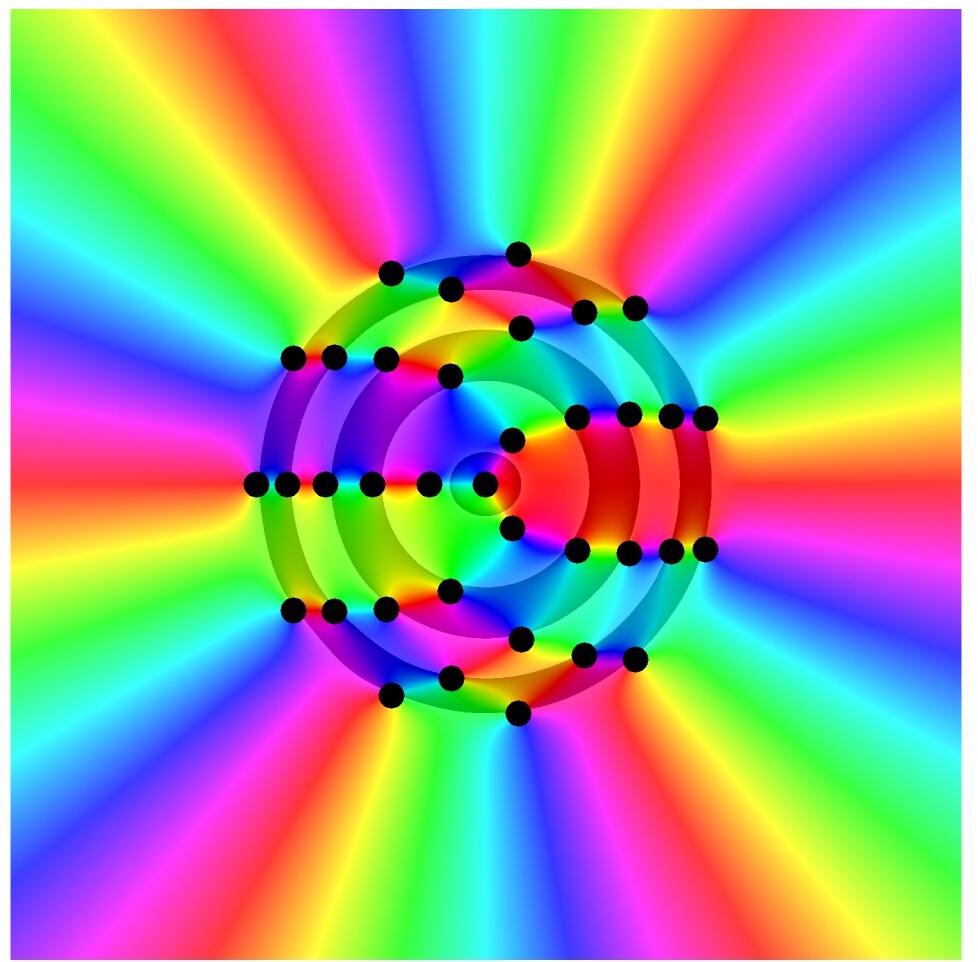}
\hfill
\includegraphics[width = .327\linewidth]{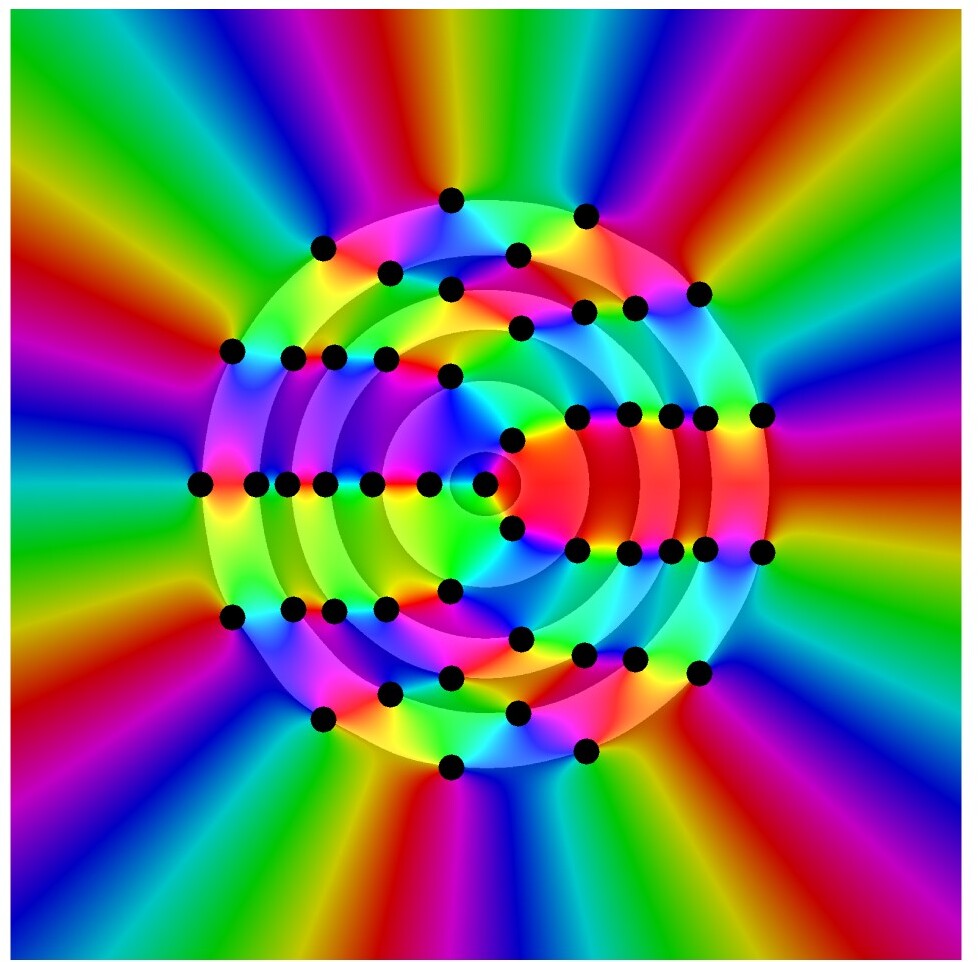}

\smallskip

\includegraphics[width = .327\linewidth]{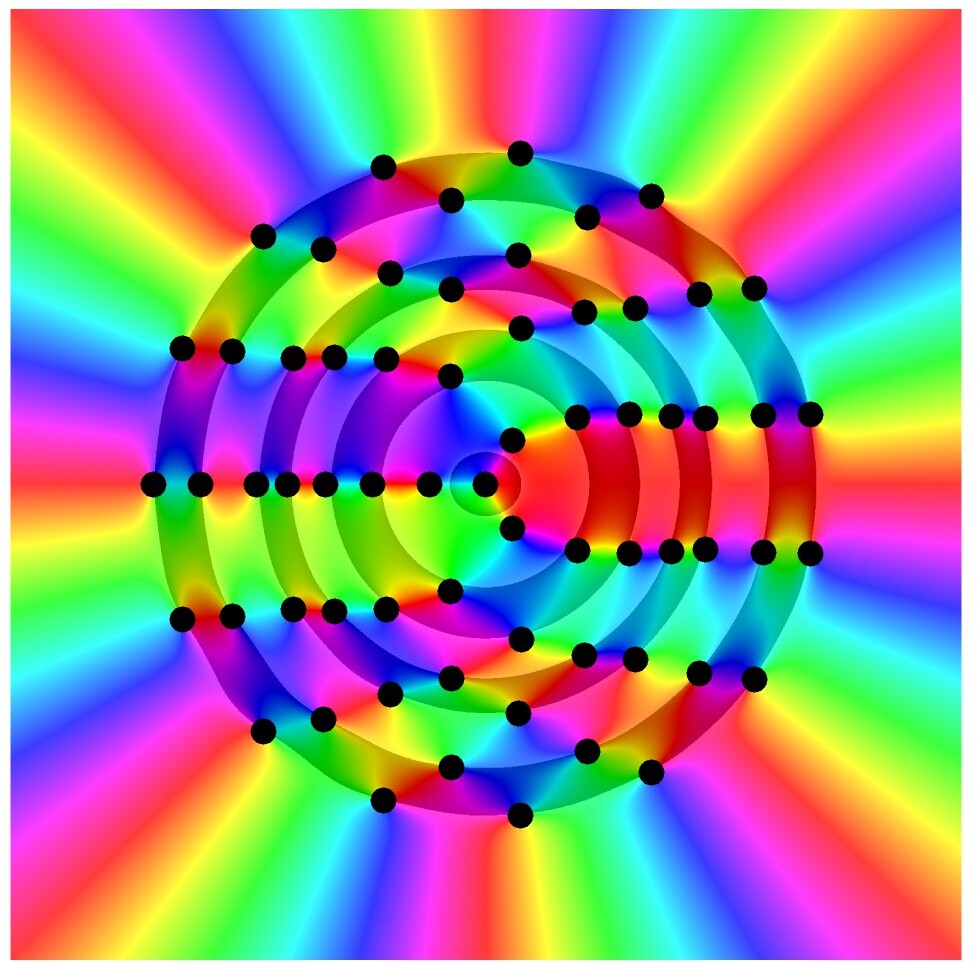}
\hfill
\includegraphics[width = .327\linewidth]{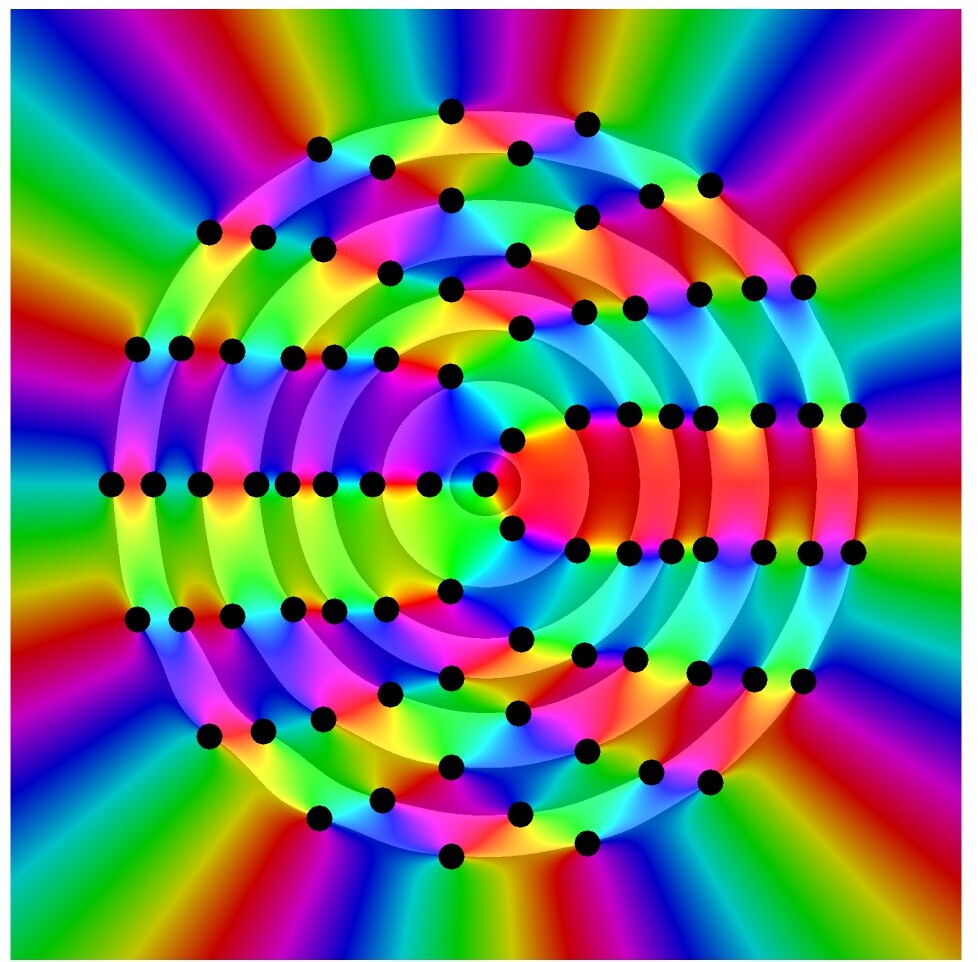}
\hfill
\includegraphics[width = .327\linewidth]{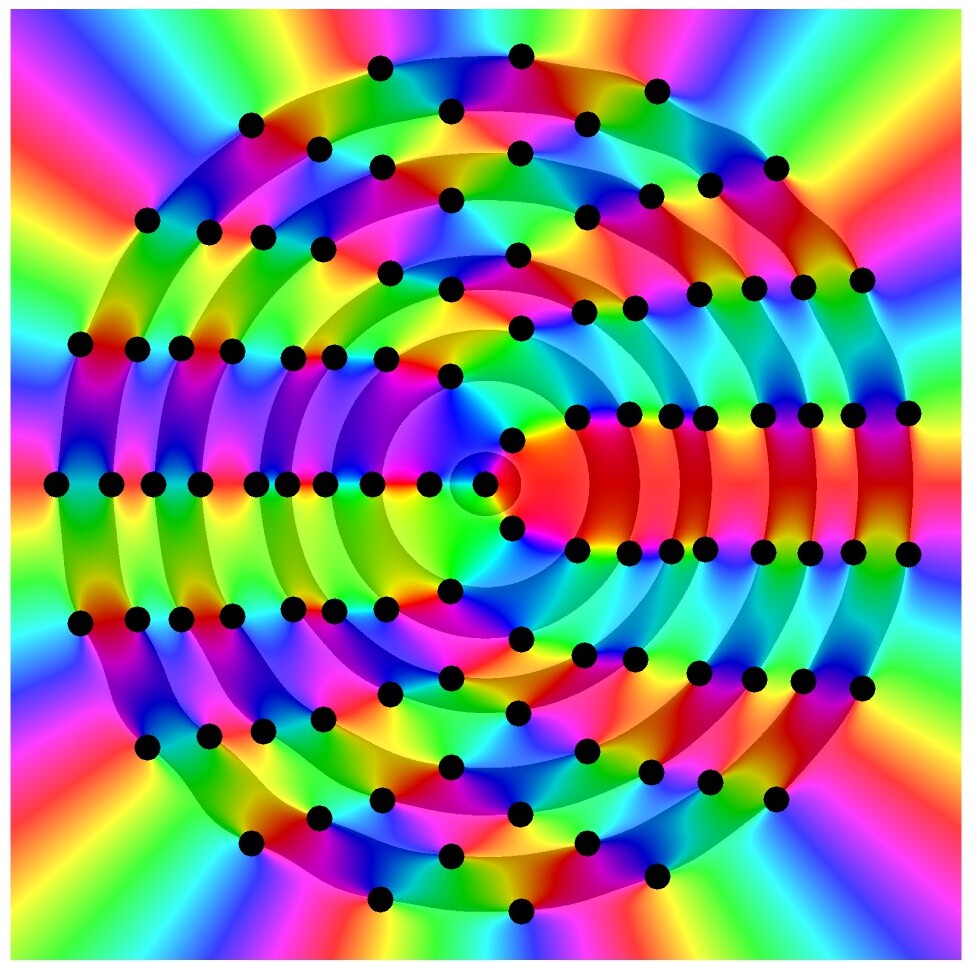}

}
\caption{Phase plots of $P_n \circ \phi$ for $n = 2, \dots, 10$ in 
Example~\ref{exp:extremal}.}
\label{fig:extremal_construction}
\end{figure}

\begin{figure}[t]
{\centering
\includegraphics[width = \linewidth]{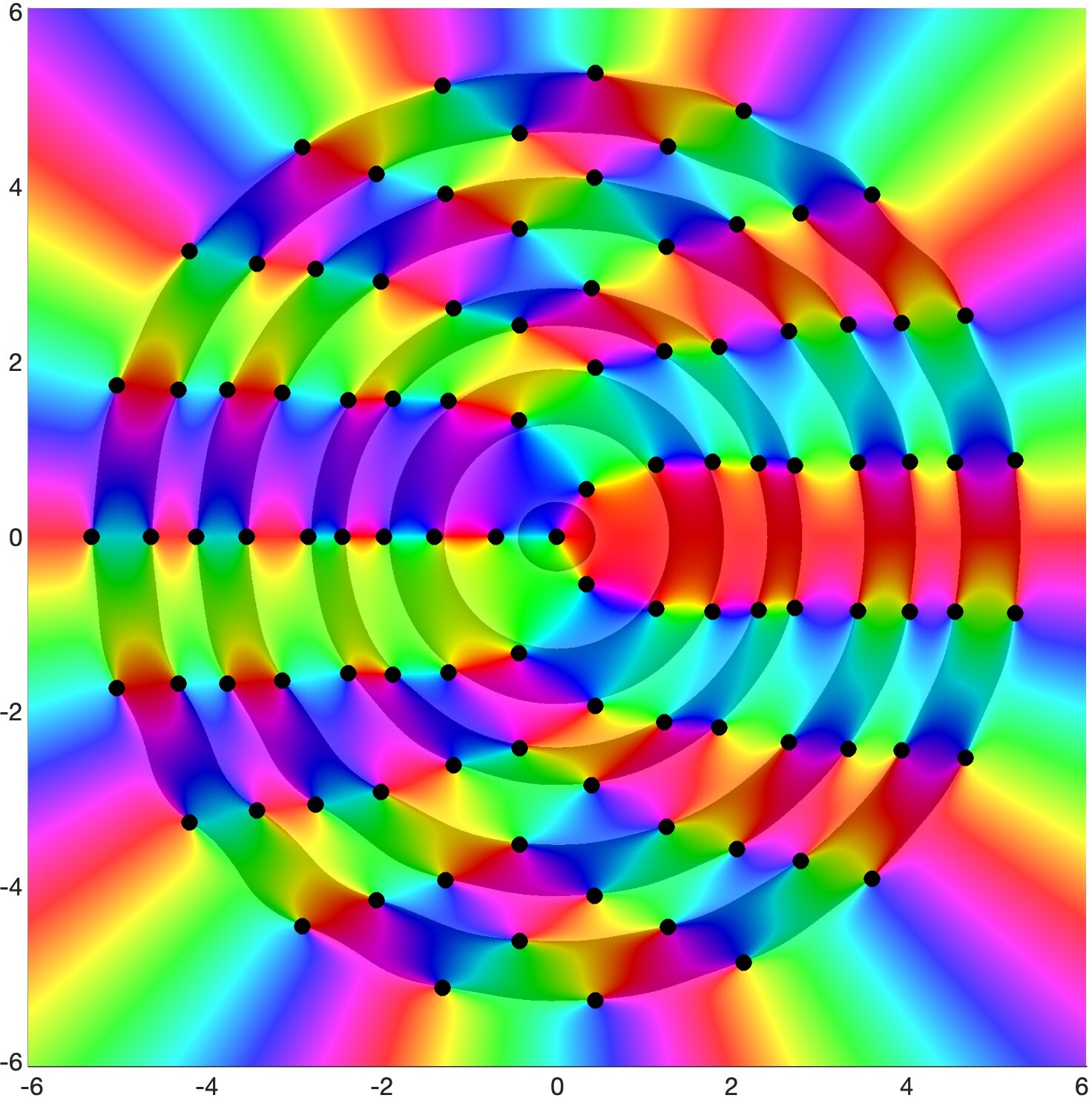}

}
\caption{Phase plot of $P_{10} \circ \phi$ in 
Example~\ref{exp:extremal}.}
\label{fig:extremal_final}
\end{figure}

\begin{example}\label{exp:extremal}
We consider $P_n$ in~\eqref{eqn:extremal_polynomial}, $n = 2, \ldots, 10$,
with 
\begin{align*}
a_1 &= 1, & a_2 &= 1, & a_3 &= 10^{-1}, & a_4 &= 10^{-3}, &  a_5 &= 10^{-6}, \\
a_6 &= 10^{-10}, & a_7 &= 10^{-16}, & a_8 &= 10^{-24}, & a_9 &= 10^{-34}, & 
a_{10} &= 10^{-47},
\end{align*}
and visualize them with phase plots; see Section~\ref{sect:location} for 
details.
Since the coefficients $a_k$ decrease exponentially, the radii 
$\abs{a_{k-1}}/\abs{a_k}$
in~\eqref{eqn:approx_zeros} become very large. 
To nevertheless visualize all zeros and critical curves of $P_n$ in the same 
plot, we display the phase of $P_n \circ \phi$ with the 
sense-preserving homeomorphism $\phi : \C \to \C$, $\phi(z) = z \exp(\abs 
z^2)$, instead of $P_n$, contracting the complex plane.
Figure~\ref{fig:extremal_construction} shows phase plots of $P_n \circ \phi$ 
for $n = 2, \dots, 10$, each with $n^2$ zeros, and illustrates the iterative 
process described in the proof of Theorem~\ref{thm:extremal_number}; see also
Figure~\ref{fig:extremal_final} for $n = 10$.
The approximate circles with the zeros of $P_n$ and the critical curves, as 
well as the distorted annuli on which $P_n$ has the same orientation, are 
apparent in all plots.
For the numerical computation of the zeros, we used our MATLAB-implementation 
of the (harmonic) Newton method~\cite{SeteZur2020} with initial points 
$\zeta_{k,j}$ in~\eqref{eqn:approx_zeros}.  The computation yields 
approximations $z_{k,j}$ of the $100$ zeros of $P_{10}$ with a very small
maximum relative residual
\begin{equation*}
\max_{k = 1, \dots, 10} \: \max_{j=1, \ldots, 2k-1} \abs*{\frac{P_{10}(z_{k,j}, 
\conj{z}_{k,j})}{a_k z_{k,j}^k}} = 1.4151 \cdot 10^{-15},
\end{equation*}
which is close to machine precision.
\end{example}

\begin{remark}
It would be interesting to study the minimum number of nonzero coefficients 
$\alpha_{j,k-j}$ that are needed, such that the polyanalytic 
polynomial~\eqref{eqn:papolynomial} has the maximum number of zeros.  In 
general, a polyanalytic polynomial of degree $n$ has $(n+1)(n+2)/2$ 
coefficients.  While the extremal polynomial of 
Wilmshurst~\cite[p.~2080]{Wilmshurst1998} has $2n$ nonzero coefficients, 
and the one of Bshouty, Hengartner, and Suez~\cite{BshoutyHengartnerSuez1995} 
has $n+1$ nonzero coefficients,
our construction in~\eqref{eqn:extremal_polynomial} requires $n$ nonzero 
coefficients, and leads to a more sparse polynomial.  We conjecture that 
$n$ is the minimum number of nonzero coefficients such that a polyanalytic 
polynomial of degree $n$ can have the maximum number of $n^2$ zeros.
\end{remark}
 \section{Wilmshurst's problem}
\label{sect:wilmshurst}

By the fundamental theorem of algebra, every analytic polynomial of degree $n 
\geq 0$ has exactly $n$ zeros counted with multiplicities.  The search for a 
similar result on the number of zeros of harmonic polynomials $p(z) + 
\conj{q(z)}$, which takes the individual degrees of $p$ and $q$ into account, 
was initiated by Sheil-Small~\cite[p.~19]{Sheil-Small1992} in 1992;
see also~\cite[p.~50~ff.]{Sheil-Small2002}.

Since then, it has become a highly active research topic over the past 
decades.
Among the most important results are the publications
\cite{Geyer2008,KhavinsonLundbergPerry2024,KhavinsonSwiatek2003, 
LeeLerarioLundberg2015,Lundberg2023,PeretzSchmid1997,SeteZur2021,
Wilmshurst1998}.
In this section, we discuss a possible extension to polyanalytic polynomials. 

For $n > m \ge 1$, denote
\begin{equation}
\cH_{n,m} = \{p(z) + \conj{q(z)} : \deg(p) = n, \deg(q) = m\}.
\end{equation}
By Corollary~\ref{cor:wilmshurst}, the number of zeros of $P \in \cH_{n,m}$, 
denoted by $N(P)$, satisfies $N(P) \leq n^2$.
Hence, the \emph{maximal valence} of harmonic polynomials,
\begin{equation}
V(\cH_{n,m}) = \max_{P \in \cH_{n,m}} N(P) \le n^2,
\end{equation}
is finite.
As mentioned in the introduction, Wilmshurst's 
conjecture~\eqref{eq:wilmshurst_conjecture} is in general not true.
We therefore refer to bounding $V(\cH_{n,m})$ in terms of $n$ and $m$ as 
\emph{Wilmshurst's problem}.
Note that Lee, Lerario, and Lundberg, who 
disproved~\eqref{eq:wilmshurst_conjecture}, conjectured instead
\begin{equation}
V(\cH_{n,m}) \le n + 2m(n-1),
\end{equation}
which is larger than $n^2$ for $m > \frac{n}{2}$, and hence, even if true, 
not optimal; see~\cite[Conj.~1.4]{LeeLerarioLundberg2015} and 
also~\cite{Lundberg2023}.

Next, we consider for $n > m \ge 1$ the set of polyanalytic polynomials
\begin{equation}
\cP_{n,m} = \{ p(z) + Q(z,\conj z) : \deg(p) = n, \deg(Q) = m\}.
\end{equation}
In contrast to $\cH_{n,m}$, nonzero mixed terms $\alpha_{j,k-j} z^j 
\conj{z}^{k-j}$ 
with $k \leq m$ in~\eqref{eqn:papolynomial} are allowed.  Any $P \in 
\cP_{n,m}$ has at most $n^2$ zeros by Corollary~\ref{cor:wilmshurst}, i.e., 
also the maximal valence of $\cP_{n,m}$, 
\begin{equation} \label{eqn:max_valence_Pnm}
V(\cP_{n,m}) = \max_{P \in \cP_{n,m}} N(P) \le n^2,
\end{equation}
is finite.
Moreover, since $\cH_{n,m} \subseteq \cP_{n,m}$, we have
\begin{equation}\label{eqn:VP}
V(\cH_{n,m}) \le V(\cP_{n,m}), \quad n > m \geq 1.
\end{equation}
For which $n$ and $m$ does equality in~\eqref{eqn:VP} hold?  The next 
theorem states that we have $V(\cH_{n,m}) = V(\cP_{n,m})$ for $m = 1$ and $m = 
n-1$, which are the only cases where $V(\cH_{n,m})$ is known.

\begin{theorem}\label{thm:VP}
Let $n \ge 2$. Then the following holds:
\begin{enumerate}
\item \label{it:VP_1}
$V(\cH_{n,1}) = V(\cP_{n,1}) = 3n-2$,

\item \label{it:VP_n-1} $V(\cH_{n,n-1}) = V(\cP_{n,n-1}) = n^2$.
\end{enumerate}
\end{theorem}

\begin{proof}
\ref{it:VP_1}
Let $P \in \cP_{n,1}$, i.e., $P(z,\conj z) = p(z) + \alpha_{1,0} z + 
\alpha_{0,1}\conj z + \alpha_{0,0}$.  If $\alpha_{0,1} = 0$, then $P$ is an 
analytic polynomial, and hence has at most $n$ distinct zeros.  If 
$\alpha_{0,1} \neq 0$, then consider $Q(z,\conj z) = \alpha_{0,1}^{-1} P(z, 
\conj{z}) = \widetilde{p}(z) + \conj{z}$.  The 
number of zeros of $Q$, and therefore also of $P$, is at most $3n-2$ 
by~\cite{KhavinsonSwiatek2003} and this bound is
attained for some harmonic polynomial~\cite{Geyer2008}, which 
shows~\ref{it:VP_1}.

\ref{it:VP_n-1}
By~\eqref{eqn:max_valence_Pnm} and~\eqref{eqn:VP}, we have $V(\cH_{n,n-1}) \leq 
V(\cP_{n,n-1}) \leq n^2$.
Moreover, there exists a $P \in \cH_{n,n-1}$ with $n^2$ zeros, e.g., the 
polynomial in Theorem~\ref{thm:extremal_number} or Wilmshurst's 
polynomial~\cite[p.~2080]{Wilmshurst1998}.  This establishes~\ref{it:VP_n-1}.
\end{proof}

We can represent harmonic and polyanalytic polynomials as trigonometric 
polynomials of order $n$ by writing $z = r \ee^{\ii \varphi}$ with $r \geq 0$ 
and $\varphi \in \R$.
A harmonic polynomial then has the form
\begin{equation}
P(z, \conj{z}) = p(z) + \conj{q(z)}
= \sum_{k=0}^n \beta_k z^k + \sum_{k=1}^n \beta_{-k} \conj{z}^k
= \sum_{k=-n}^n \beta_k r^{\abs{k}} \ee^{\ii k \varphi},
\end{equation}
see also~\cite[Sect.~2.6.9]{Sheil-Small2002}, while a polyanalytic polynomial 
has the form
\begin{equation}
P(z, \conj{z}) = \sum_{k=0}^n \sum_{j=0}^k \alpha_{j,k-j} z^j \conj{z}^{k-j}
= \sum_{k=0}^n \sum_{j=0}^k \alpha_{j,k-j} r^k \ee^{\ii (2j-k) \varphi}.
\end{equation}
For a harmonic polynomial the coefficient of $\ee^{\ii k \varphi}$ is $\beta_k 
r^{\abs{k}}$, while
for the polyanalytic polynomial the coefficient of $\ee^{\ii k \varphi}$
is a polynomial $p_k$ in $r$ of degree at most $\abs{k}$.
Note that $p_k$ is even if $k$ is even and odd if $k$ is odd.
Despite this difference, the authors would not be surprised if $V(\cH_{n,m}) = 
V(\cP_{n,m})$ holds for all $n > m \geq 1$,
i.e., that the absence of mixed terms $\alpha_{j,k-j} z^j \conj{z}^{k-j}$ does 
not decrease the maximal valence.
Proving an upper bound on the maximum number of zeros of harmonic polynomials, 
and thus further advancing the work of Sheil-Small and others, may not rely 
exclusively on harmonicity instead of polyanalyticity.

\paragraph{Conflict of interest.}
The authors have no competing interests to declare that are relevant to the 
content of this article.

\small
\bibliography{sete_zur_2024}
\bibliographystyle{siam} 
\normalsize

\end{document}